\documentclass[11pt, twoside, leqno]{article}
\usepackage{amsfonts}

\usepackage{amssymb}
\usepackage{amsmath}
\usepackage{amsthm}
\usepackage{xcolor}
\usepackage{mathrsfs}

\allowdisplaybreaks

\def\rr{{\mathbb R}}

\def\rd{{{\rr}^d}}
\def\zz{{\mathbb Z}}
\def\cc{{\mathbb C}}
\def\bb{{\mathbb B}}
\def\nn{{\mathbb N}}

\def\cx{{\mathcal X}}

\def\fz{\infty}
\def\az{\alpha}
\def\bz{\beta}
\def\dz{\delta}
\def\bdz{\Delta}
\def\ez{\epsilon}
\def\gz{{\gamma}}

\def\kz{\kappa}
\def\lz{\lambda}

\def\oz{\omega}

\def\tz{\theta}
\def\sz{\sigma}
\def\vz{\varphi}

\def\lf{\left}
\def\r{\right}

\def\hs{\hspace{0.25cm}}
\def\ls{\lesssim}
\def\gs{\gtrsim}

\def\noz{\nonumber}
\def\wz{\widetilde}

\def\st{\subset}

\def\diam{\mathop\mathrm{diam}}
\def\supp{\mathop\mathrm{supp}}

\def\dint{\displaystyle\int}

\def\lp{{L^p(\mu)}}
\def\lq{{L^q(\mu)}}

\def\lpo{{L^{p_1}(\mu)}}
\def\lqo{{L^{q_1}(\mu)}}
\def\lon{{L^1(\mu)}}

\def\li{{L^{\infty}(\mu)}}
\def\rbmo{\mathop\mathrm{\,{\rm RBMO}(\mu)}}
\def\bph{L^{\Phi}(\mu)}
\def\bps{L^{\Psi}(\mu)}
\def\hon{{H^1(\mu)}}
\def\hp{H_{\rm{atb}}^{1,p}(\mu)}
\def\osc{{\rm Osc}_{\exp L^r}(\mu)}
\def\osi{{\rm Osc}_{\exp L^{r_i}}(\mu)}
\def\osj{{\rm Osc}_{\exp L^{r_j}}(\mu)}
\def\ep{\exp L^r,\,B,\,\mu/\mu(2B)}

\def\loc{{\mathop\mathrm{loc\,}}}

\newtheorem{theorem}{Theorem}[section]
\newtheorem{lemma}[theorem]{Lemma}
\newtheorem{corollary}[theorem]{Corollary}
\newtheorem{proposition}[theorem]{Proposition}
\theoremstyle{definition}
\newtheorem{remark}[theorem]{Remark}
\newtheorem{definition}[theorem]{Definition}

\numberwithin{equation}{section}

\begin{document}

\arraycolsep=1pt

\title{\Large\bf Generalized Fractional Integrals and Their Commutators over
Non-homogeneous Metric Measure Spaces \footnotetext {\hspace{-0.35cm}
2010 {\it Mathematics Subject Classification}. Primary 47B06; Secondary 47B47,
42B25, 42B35, 30L99. \endgraf {\it Key words and phrases}. non-homogeneous metric measure space,
fractional integral, commutator, Orlicz space, Hardy space,
${\mathop\mathrm{\,RBMO(\mu)}}$, ${\mathrm{Osc}_{\exp L^r}(\mu)}$.
\endgraf This project is supported by the National
Natural Science Foundation of China (Grant Nos. 11171027, 11361020 \& 11101038), 
the Specialized Research Fund for the Doctoral Program of Higher Education
of China (Grant No. 20120003110003) and the Fundamental Research Funds for Central
Universities of China (Grant No. 2012LYB26).}}
\author{Xing Fu, Dachun Yang\,\footnote{Corresponding author}\ \ and Wen Yuan}
\date{}
\maketitle

\vspace{-0.8cm}

\begin{center}
\begin{minipage}{13cm}
{\small {\bf Abstract}\quad Let $({\mathcal X},d,\mu)$ be a metric measure space
satisfying both the upper doubling and the geometrically doubling
conditions. In this paper, the authors establish some
equivalent characterizations for the boundedness of fractional integrals over
$({\mathcal X},d,\mu)$. The authors
also prove that multilinear commutators of fractional integrals
with ${\mathop\mathrm{\,RBMO(\mu)}}$ functions are bounded
on Orlicz spaces over $({\mathcal X},d,\mu)$, which include
Lebesgue spaces as special cases. The weak type
endpoint estimates for multilinear commutators of fractional integrals with
functions in the Orlicz-type space ${\mathrm{Osc}_{\exp L^r}(\mu)}$,
where $r\in [1,\infty)$, are
also presented. Finally, all these results are applied
to a specific example of fractional
integrals over non-homogeneous metric measure spaces.}
\end{minipage}
\end{center}


\section{Introduction}\label{s1}

\hskip\parindent During the past ten to fifteen years, considerable
attention has been paid to the study of the classical theory of
harmonic analysis on Euclidean spaces with non-doubling measures only satisfying
the polynomial growth condition (see, for example,
\cite{gm,gg,t01a,t01b,t03a,t03b,t04,t05,cs,ntv,hmy1,hmy2,hmy3,hmy4,cmy,yy12}).
Recall that a Radon measure $\mu$ on $\mathbb{R}^{d}$ is said to only satisfy
the polynomial growth condition, if there exists a positive constant $C$ such that,
for all $x\in\mathbb{R}^{d}$ and $r\in (0,\fz)$,
\begin{equation}\label{1.1}
\mu(B(x,r))\le Cr^{\kz},
\end{equation}
where $\kz$ is some fixed number in $(0,d]$ and $B(x,r):=\{y\in\rr^d:\,|y-x|<r\}$.
The analysis associated with such non-doubling measures $\mu$ as in
\eqref{1.1} has proved to play a striking role in solving
the long-standing open Painlev\'e's problem
and Vitushkin's conjecture by Tolsa \cite{t03b, t04, t05}.

Obviously, the non-doubling measure $\mu$ as in \eqref{1.1} may not satisfy
the well-known doubling condition, which is a key assumption
in harmonic analysis on spaces of homogeneous type in the sense of Coifman and Weiss
\cite{cw71, cw77}.
To unify both spaces of homogeneous type and
the metric spaces endowed with measures only satisfying the polynomial growth
condition, Hyt\"onen \cite{h10} introduced
a new class of metric measure spaces satisfying both the so-called
geometrically doubling and the upper doubling conditions (see also, respectively,
Definitions \ref{d1.1} and \ref{d1.3} below), which are called
\emph{non-homogeneous metric measure spaces}.
Recently, many classical results have been proved still valid
if the underlying spaces are replaced by the non-homogeneous metric measure spaces
(see, for example, \cite{h10,ly1,bd,hlyy,hm,hyy,lyy,fyy,ly}).
It is now also known that the theory of the singular integral operators
on non-homogeneous metric measure spaces arises naturally in the study of
complex and harmonic analysis questions in several complex variables
(see \cite{vw09,hm}). More progresses on the Hardy space $H^1$ and boundedness
of operators on non-homogeneous metric measure spaces can be found in the survey
\cite{yyf} and the monograph \cite{yyh}.

Let $(\cx,d,\mu)$ be a non-homogeneous metric measure space
in the sense of Hyt\"onen \cite{h10}. In this paper, we establish some
equivalent characterizations for the boundedness of fractional integrals over
$({\mathcal X},d,\mu)$. We also prove that multilinear commutators of
fractional integrals with ${\mathop\mathrm{\,RBMO(\mu)}}$ functions are bounded
on Orlicz spaces over $({\mathcal X},d,\mu)$, which include
Lebesgue spaces as special cases. The weak type
endpoint estimates for multilinear commutators of fractional integrals with
functions in the Orlicz-type space ${\mathrm{Osc}_{\exp L^r}(\mu)}$,
where $r\in [1,\fz)$, are also presented. Finally, all these results are applied
to a specific example of fractional integrals over non-homogeneous metric measure spaces.
The results of this paper round out the picture on fractional
integrals and their commutators over
non-homogeneous metric measure spaces.

Recall that the well-known Hardy-Littlewood-Sobolev theorem
(see, for example, \cite[pp. 119-120, Theorem 1]{s70}) states that
the classical fractional integral $I_\az$, with $\az\in (0,d)$,
is bounded from $L^p(\rd)$ into $L^q(\rd)$, for all $p\in (1,d/\az)$
and $1/q=1/p-\az/d$, and bounded from $L^1(\rd)$ to weak $L^{d/(d-\az)}(\rd)$.
Chanillo \cite{c} further showed that the commutator $[b,I_\az]$, generated by
$b\in {\mathop\mathrm{BMO}}(\rd)$ and $I_\az$,  which is defined by
$$[b,I_\az](f)(x):=b(x)I_\az(f)(x)-I_\az(bf)(x),\quad x\in\rd,$$
is bounded from $L^p(\rd)$ into $L^q(\rd)$ for all $\az\in (0,d)$,
$p\in (1,d/\az)$ and $1/q=1/p-\az/d$.
These results, when the $d$-dimensional Lebesgue measure
is replaced by the non-doubling measure $\mu$ as in \eqref{1.1},
were obtained by Garc\'ia-Cuerva and Martell \cite{gm}
and by Chen and Sawyer \cite{cs}, respectively.
Moreover, also in this setting with the non-doubling measure $\mu$ as in \eqref{1.1},
some equivalent characterizations for the boundedness of fractional integrals
were established in \cite{hmy4} and the boundedness for
the multilinear commutators of fractional integrals with $\rbmo$
or ${\mathrm{Osc}_{\exp L^r}(\mu)}$ functions was presented in \cite{hmy1}.

On the other hand, due to the request of applications,
as a natural extension of Lebesgue spaces, the Orlicz space was introduced by Birnbaum-Orlicz
in \cite{bo31} and Orlicz in \cite{o32}.
Since then, the theory of Orlicz spaces and its applications
have been well developed (see, for example, \cite{rr,rr02,m}).

To state the main results of this paper, we first recall some necessary
notions.

The following notion of the geometrically doubling
is well known in analysis on metric spaces, which was originally introduced
by Coifman and Weiss in \cite[pp.\,66-67]{cw71} and is also
known as \emph{metrically doubling} (see, for example, \cite[p.\,81]{he}).

\begin{definition}\label{d1.1}
A metric space $(\cx,d)$ is said to be \emph{geometrically doubling} if there
exists some $N_0\in \nn$ such that, for any ball
$B(x,r)\st \cx$, there exists a finite ball covering $\{B(x_i,r/2)\}_i$ of
$B(x,r)$ such that the cardinality of this covering is at most $N_0$.
\end{definition}

\begin{remark}\label{r1.2}
Let $(\cx,d)$ be a metric space. In \cite{h10}, Hyt\"onen showed that
the following statements are mutually equivalent:
\vspace{-0.25cm}
\begin{itemize}
  \item[\rm(i)] $(\cx,d)$ is geometrically doubling.
\vspace{-0.25cm}
  \item[\rm(ii)] For any $\ez\in (0,1)$ and any ball $B(x,r)\st \cx$,
there exists a finite ball covering $\{B(x_i,\ez r)\}_i$ of
$B(x,r)$ such that the cardinality of this covering is at most $N_0\ez^{-n}$,
here and in what follows, $N_0$ is as in Definition \ref{d1.1} and
$n:=\log_2N_0$.
\vspace{-0.25cm}
  \item[\rm(iii)] For every $\ez\in (0,1)$, any ball $B(x,r)\st \cx$ contains
at most $N_0\ez^{-n}$ centers of disjoint balls $\{B(x_i,\ez r)\}_i$.
\vspace{-0.25cm}
  \item[\rm(iv)] There exists $M\in \nn$ such that any ball $B(x,r)\st \cx$
  contains at most $M$ centers $\{x_i\}_i$ of disjoint balls $\{B(x_i, r/4)\}_{i=1}^M$.
  \end{itemize}
\end{remark}

Recall that spaces of homogeneous type are geometrically doubling, which was proved
by Coifman and Weiss in \cite[pp.\,66-68]{cw71}.

The following notion of upper doubling metric measure spaces was originally introduced
by Hyt\"onen \cite{h10} (see also \cite{hlyy,lyy}).

\begin{definition}\label{d1.3}
A metric measure space $(\cx,d,\mu)$ is said to be \emph{upper doubling} if
$\mu$ is a Borel measure on $\cx$ and there exist a \emph{dominating function}
$\lz:\cx \times (0,\fz)\to (0,\fz)$ and a positive constant $C_{\lz}$,
depending on $\lz$, such that, for each $x\in \cx$, $r\to \lz(x,r)$ is
non-decreasing and, for all $x\in \cx$ and $r\in (0,\fz)$,
\begin{equation}\label{1.2}
\mu(B(x,r))\le\lz(x,r)\le C_{\lz}\lz(x,r/2).
\end{equation}
A metric measure space $(\cx,d,\mu)$ is called a \emph{non-homogeneous metric measure space}
if $(\cx,d)$ is geometrically doubling and $(\cx,d,\mu)$ upper doubling.
\end{definition}

\begin{remark}\label{r1.4}
(i) Obviously, a space of homogeneous type is a
special case of upper doubling spaces, where we take the dominating function
$\lz(x,r):=\mu(B(x,r))$. On the other hand, the Euclidean space
$\rd$ with any Radon measure $\mu$ as in \eqref{1.1} is also an upper doubling
space by taking the dominating function $\lz(x,r):=C_0r^{\kz}$.

(ii) Let $(\cx,d,\mu)$ be upper doubling with $\lz$ being the dominating
function on $\cx \times (0,\fz)$ as in Definition \ref{d1.3}. It was proved
in \cite{hyy} that there exists another
dominating function $\wz{\lz}$ such that $\wz{\lz}\le \lz$, $C_{\wz{\lz}}\le C_{\lz}$
and, for all $x,\,y\in \cx$ with $d(x,y)\le r$,
\begin{equation}\label{1.3}
\wz{\lz}(x,r)\le C_{\wz{\lz}}\wz{\lz}(y,r).
\end{equation}

(iii) It was shown in \cite{tl} that the upper doubling condition
is equivalent
to the \emph{weak growth condition}:
there exist a dominating function $\lz:\cx\times(0,\fz)\to(0,\fz)$,
with $r\to\lz(x,r)$ non-decreasing, positive constants $C_\lz$, depending on $\lz$,
and $\ez$ such that
\begin{itemize}
\item[(a)] for all $r\in(0,\fz)$, $t\in[0,r]$, $x,\,y\in\cx$ and $d(x,y)\in[0,r]$,
$$|\lz(y,r+t)-\lz(x,r)|\le C_\lz\lf[\frac{d(x,y)+t}r\r]^{\ez}
\lz(x,r);$$

\item[(b)] for all $x\in\cx$ and $r\in(0,\fz)$,
$$\mu(B(x,r))\le\lz(x,r).$$
\end{itemize}
\end{remark}

Based on Remark \ref{r1.4}(ii), from now on, we \emph{always assume that
$(\cx,d,\mu)$ is a non-homogeneous metric measure space with the dominating function $\lz$
satisfying \eqref{1.3}}.

We now recall the notion of the coefficient $K_{B,S}$ introduced by Hyt\"onen \cite{h10}, which is analogous
to the quantity $K_{Q,R}$ introduced by Tolsa \cite{t01b,t03a}.
It is well known that $K_{B,S}$ well characterizes the geometrical properties of
balls $B$ and $S$.

\begin{definition}\label{d1.5}
For any two balls $B\st S$, define
$$K_{B,S}
:=1+\int_{2S\setminus B}\frac1{\lz (c_B,d(x,c_B))}\,d\mu(x),$$
where $c_B$ is the center of the ball $B$.
\end{definition}

\begin{remark}\label{r1.6}
The following discrete version, ${\wz K}_{B,S}$,
of $K_{B,S}$ defined in Definition \ref{d1.5},
was first introduced by Bui and Duong \cite{bd}
in non-homogeneous metric measure spaces, which is more close to the quantity $K_{Q,R}$
introduced by Tolsa \cite{t01a} in the setting of non-doubling measures.
For any two balls $B\st S$, let ${\wz K}_{B,S}$ be defined by
\begin{equation*}
{\wz K}_{B,S}
:=1+\sum_{k=1}^{N_{B,S}}\frac{\mu(6^{k}B)}{\lz(c_{B},6^{k}r_{B})},
\end{equation*}
where $r_B$ and $r_S$ respectively denote the \emph{radii} of the balls
$B$ and $S$, and $N_{B,S}$ the \emph{smallest integer} satisfying
$6^{N_{B,S}}r_{B}\ge r_{S}$. Obviously, $K_{B,S}\ls {\wz K}_{B,S}$.
As was pointed by Bui and Duong \cite{bd}, in general,
it is not true that $K_{B,S}\sim{\wz K}_{B,S}$.
\end{remark}

Though the measure doubling condition is not assumed uniformly for all balls
in the non-homogeneous metric measure space $(\cx,d,\mu)$, it was shown in \cite{h10} that there exist still many
balls which have the following $(\eta,\bz)$-doubling property.

\begin{definition}\label{d1.7}
Let $\eta,\,\bz\in (1,\fz)$. A ball $B\st \cx$ is said to be
\emph{$(\eta,\bz)$-doubling} if $\mu(\eta B)\le \bz\mu(B)$.
\end{definition}

To be precise, it was proved in \cite[Lemma 3.2]{h10} that, if a metric measure
space $(\cx,d,\mu)$ is upper doubling and $\eta,\,\bz\in(1,\fz)$ satisfying
$\bz>C_{\lz}^{\log_2\eta}=:\eta^\nu$,
then, for any ball $B\st \cx$, there exists some $j\in \zz_+:=\nn\cup \{0\}$ such
that $\eta^jB$ is $(\eta,\bz)$-doubling. Moreover, let $(\cx,d)$ be geometrically
doubling, $\bz>\eta^n$ with $n:=\log_2N_0$ and $\mu$ a Borel measure on $\cx$
which is finite on bounded sets. Hyt\"onen \cite[Lemma 3.3]{h10} also showed that, for
$\mu$-almost every $x\in \cx$, there exist arbitrary small $(\eta,\bz)$-doubling
balls centered at $x$. Furthermore, the radii of these balls may be chosen to be
the form $\eta^{-j}B$ for $j\in \nn$ and any preassigned number $r\in (0,\fz)$.
Throughout this paper, for any $\eta\in (1,\fz)$ and ball $B$, the \emph{smallest
$(\eta,\bz_\eta)$-doubling ball of the form $\eta^j B$ with $j\in \nn$} is denoted by
$\wz B^\eta$, where
\begin{equation}\label{1.4}
\bz_\eta:=\max\{\eta^{3n},\,\eta^{3\nu}\}+30^n+30^\nu
=\eta^{3(\max\{n,\nu\})}+30^n+30^\nu.
\end{equation}
In what follows, by a \emph{doubling ball} we mean a
$(6,\bz_6)$-doubling ball and $\wz B^6$ is \emph{simply denoted} by $\wz B$.

Now we recall the following notion of $\rbmo$ from \cite{h10}.

\begin{definition}\label{d1.8}
Let $\rho\in (1,\fz)$. A function $f\in
L_{\rm{loc}}^{1}(\mu)$ is said to be in the \emph{space $\rbmo$}
if there exist a positive constant $C$ and, for any
ball $B\subset \cx$, a number $f_B$ such that
\begin{equation}\label{1.5}
\frac{1}{\mu(\rho B)}\int_{B}\lf|f(x)-f_B\r|\,d\mu(x)\le C
\end{equation}
and, for any two balls $B$ and $B_1$ such that $B\st B_1$,
\begin{equation}\label{1.6}
|f_B-f_{B_1}|\le CK_{B,B_1}.
\end{equation}
The infimum of the positive constants $C$ satisfying both \eqref{1.5} and \eqref{1.6}
is defined to be the \emph{$\rm{RBMO}(\mu)$ norm} of $f$ and denoted by $\|f\|_{\rbmo}$.
\end{definition}

From \cite[Lemma 4.6]{h10}, it follows that the space $\rbmo$
is independent of the choice of $\rho\in (1,\fz)$.

In this paper, we consider a variant of the generalized fractional integrals from
\cite[Definition 4.1]{gg} (see also \cite[(1.4)]{hmy4}).

\begin{definition}\label{d1.9}
Let $\az\in (0,1)$.
A function $K_{\az}\in L_{\rm{loc}}^{1}(\cx\times
\cx\setminus\{(x,y):x=y\})$ is called a \emph{generalized fractional integral
kernel} if there exists a positive constant $C_{K_\az}$, depending on $K_\az$,
such that

(i) for all $x,\,y\in\cx$ with $x\ne y$,
\begin{equation}\label{1.7}
|K_{\az}(x,y)|\le C_{K_\az}\frac{1}{[\lz(x,d(x,y))]^{1-\az}};
\end{equation}

(ii) there exist positive constants $\dz\in (0,1]$ and $c_{K_\az}\in(0,\fz)$ such that,
for all $x,\,\wz x,\,y\in\cx$ with $d(x,y)\ge c_{K_\az}d(x,\wz{x})$,
\begin{equation}\label{1.8}
|K_{\az}(x,y)-K_{\az}(\wz x,y)|+|K_{\az}(y,x)-K_{\az}(y,\wz x)|\le
C_{K_\az}\frac{[d(x,\wz x)]^{\dz}}{[d(x,y)]^{\dz}[\lz(x,d(x,y))]^{1-\az}}.
\end{equation}

Let $L^{\fz}_b(\mu)$ be the \emph{space of all $\li$ functions with bounded support}.
A linear operator $T_{\az}$ is called a \emph{generalized fractional integral} with
kernel $K_{\az}$ satisfying \eqref{1.7} and
\eqref{1.8} if, for all $f\in L^{\fz}_b(\mu)$ and $x\not\in \supp f$,
\begin{equation}\label{1.9}
T_{\az}f(x):=\int_{\cx}K_{\az}(x,y)f(y)\,d\mu(y).
\end{equation}
\end{definition}

\begin{remark}\label{r1.10}
(i) Without loss of generality, for the simplicity, we may assume in \eqref{1.8} that
$c_{K_\az}\equiv 2$.

(ii) If a kernel $K_\az$ satisfies \eqref{1.7} and \eqref{1.8}
with $\az=0$, then $K_\az$ is called a \emph{standard kernel} and the associated
operator $T_\az$ as in \eqref{1.9} is called a \emph{Calder\'on-Zygmund operator}
on non-homogeneous metric measure spaces (see \cite[Subsetion 2.3]{hm}).

(iii) We give a specific example of the generalized fractional integrals,
which is a natural variant of the so-called ``Bergman-type" operators from
\cite[Section 2.1]{vw09} (see also \cite[Section 12]{hm} and \cite[Section 2.2]{tc}).
Let $\cx:=\bb_{2d}$ be the open unit ball in $\cc^d$. Suppose that the measure $\mu$
satisfies the upper power bound $\mu(B(x,r))\le r^m$ with $m\in(0,2d]$ except the
case when $B(x,r)\st \bb_{2d}$. However, in the exceptional case
it holds true that $r\le {\wz d}(x):=d(x,\cc^d\setminus\bb_{2d})$, where
$d(x,y):=||x|-|y||+\lf|1-\bar{x}\cdot y/|x||y|\r|$ for all
$x,\,y\in\overline{\bb}_{2d}\st\cc^d$, and hence
$\mu(B(x,r))\le\max\{[{\wz d}(x)]^m,r^m\}=:\lz(x,r)$.
By similar arguments to those used in the proofs of
\cite[Proposition 2.13]{tc} and \cite[Section 2]{hm},
we conclude that, if $\az\in(0,1)$, then the kernel
$K_{m,\az}(x,y):=(1-\bar{x}\cdot y)^{-m(1-\az)}$, $x,\,y\in\overline{\bb}_{2d}\st\cc^d$,
satisfies the conditions \eqref{1.7} and \eqref{1.8}. So, when $\az\in(0,1)$, the fractional
integral $T_{m,\az}$, associated with $K_{m,\az}$, is an example of
the generalized fractional integrals as in Definition \ref{d1.9}.
Recall that, when $\az=0$, the operator $T_{m,0}$,
associated with $K_{m,0}$,  is just the so-called ``Bergman-type" operator
(see \cite{tc,vw09,hm}).
\end{remark}

Now we recall the notion of the atomic Hardy space from \cite{hyy}.

\begin{definition}\label{d1.11}
Let $\rho\in (1,\fz)$ and $p\in (1,\fz]$. A function
$b\in L_{\rm{loc}}^{1}(\mu)$ is called a \emph{$(p,1)_\lz$-atomic block} if

(i) there exists a ball $B$ such that $\supp b\st B$;

(ii) $\int_\cx b(x)\,d\mu(x)=0$;

(iii) for any $j\in\{1,\,2\}$, there exist a function $a_{j}$ supported on ball
$B_{j}\st B$ and a number $\lz_j\in\cc$ such that
$b=\lz_1a_1+\lz_2a_2$
and
$\|a_j\|_\lp\le [\mu(\rho B_j)]^{1/p-1}K_{B_j,B}^{-1}.$
Moreover, let $|b|_{\hp}:=|\lz_1|+|\lz_2|$.

A function $f\in L^1(\mu)$ is said to belong to the \emph{atomic Hardy space} $\hp$
if there exist $(p,1)_\lz$-atomic blocks
$\{b_{i}\}_{i=1}^\fz$ such that
$f=\sum_{i=1}^{\fz} b_{i}$ in $L^1(\mu)$
and $\sum_{i=1}^{\fz}|b_{i}|_{\hp}<\fz$.
The \emph{$\hp$ norm} of $f$ is defined by
$\|f\|_{\hp}:=\inf
\{\sum_{i=1}^{\fz}|b_{i}|_{\hp}\},$
where the infimum is taken over all the possible decompositions of
$f$ as above.
\end{definition}

\begin{remark}\label{r1.12}
(i) It was proved in \cite{hyy} that, for each $p\in (1,\fz]$, the atomic
Hardy space $\hp$ is independent of the choice of $\rho$ and that, for all
$p\in (1,\fz]$, the spaces $\hp$ and $H_{\rm{atb}}^{1,\fz}(\mu)$ coincide
with equivalent norms.
Thus, in what follows, we \emph{denote $\hp$ simply by $\hon$} and, unless
explicitly pointed out, we \emph{always assume} that $\rho=2$ in Definition
\ref{d1.11}.

(ii) It was proved in \cite[Remark 1.3(ii)]{lyy} that
the atomic Hardy space introduced by Bui and Duong \cite{bd}
and the atomic Hardy space in Definition \ref{d1.11} coincide
with equivalent norms.
\end{remark}

Then we state the first main theorem of this paper.

\begin{theorem}\label{t1.13}
Let $\az\in (0,1)$ and $T_\az$ be as in \eqref{1.9} with kernel $K_\az$
satisfying \eqref{1.7} and \eqref{1.8}. Then the following statements
are equivalent:

{\rm(I)} $T_\az$ is bounded from $\lp$ into $\lq$ for all $p\in(1,\,1/\az)$ and
$1/q=1/p-\az$;

{\rm(II)} $T_\az$ is bounded from $\lon$ into $L^{1/(1-\az),\fz}(\mu)$;

{\rm(III)} There exists a positive constant $C$ such that, for all
$f\in L^{1/\az}(\mu)$ with $T_\az f$ being finite almost everywhere,
$\|T_\az f\|_{\rbmo}\le C\|f\|_{L^{1/\az}(\mu)}$;

{\rm(IV)} $T_\az$ is bounded from $\hon$ into $L^{1/(1-\az)}(\mu)$;

{\rm(V)} $T_\az$ is bounded from $\hon$ into $L^{1/(1-\az),\fz}(\mu)$.
\end{theorem}

\begin{remark}\label{r1.14}
Theorem \ref{t1.13} covers \cite[Theorem 1.1]{hmy4} by
taking $\cx:=\rr^d$, $d$ being the usual Euclidean metric and $\mu$
as in \eqref{1.1}. The difference between Theorem \ref{t1.13}
and \cite[Theorem 1.1]{hmy4} exists in that no conclusion
of Theorem \ref{t1.13} is known to be true, while all conclusions
of \cite[Theorem 1.1]{hmy4} are true.
\end{remark}

Let $\Phi$ be a \emph{convex
Orlicz function} on $[0,\fz)$, namely, a convex
increasing function satisfying $\Phi(0)=0$, $\Phi(t)>0$ for
all $t\in (0,\fz)$ and $\Phi(t)\to \fz$ as $t\to \fz$.
Let
\begin{equation}\label{1.10}
a_{\Phi}:=\inf_{t\in (0,\fz)}\frac{t\Phi'(t)}{\Phi(t)}\quad \mathrm{and}\quad
b_{\Phi}:=\sup_{t\in (0,\fz)}\frac{t\Phi'(t)} {\Phi(t)}.
\end{equation}
We refer to \cite{m} for more properties of $a_{\Phi}$ and $b_{\Phi}$.

The
\emph{Orlicz space} $\bph$ is defined to be the space of all measurable
functions $f$ on $(\cx,d,\mu)$ such that
$\int_{\cx}\Phi(|f(x)|)\,d\mu(x)<\fz$; moreover, for any $f\in\bph$,
its \emph{Luxemburg norm} in $\bph$ is defined by
\begin{equation*}
\|f\|_{\bph}:=\inf\lf\{t\in(0,\fz):\ \int_{\cx}\Phi(|f(x)|/t)\,d\mu(x)\le 1\r\}.
\end{equation*}

For any sequence $\vec{b}:=(b_{1},\ldots,b_{k})$ of functions, the \emph{multilinear
commutator} $T_{\az,\,\vec{b}}$ of the generalized fractional integral $T_\az$ with
$\vec{b}$ is defined by setting, for all suitable functions $f$,
\begin{equation}\label{1.11}
T_{\az,\,\vec{b}}f:=[b_{k},\cdots,[b_{1},T_\az]\cdots]f,
\end{equation}
where
\begin{equation}\label{1.12}
[b_{1},T_\az]f:=b_{1}T_\az f-T_\az(b_{1}f).
\end{equation}

The second main result of this paper is the following boundedness
of the multilinear commutator $T_{\az,\vec{b}}$ on Orlicz spaces.

\begin{theorem}\label{t1.15}
Let $\az\in (0,1)$, $k\in\mathbb{N}$ and $b_{j}\in
\rbmo$ for all $j\in\{1,\ldots,k\}$. Let $\Phi$ be a convex Orlicz function
and $\Psi$ defined, via its inverse, by setting, for all $t\in (0,\fz)$,
$\Psi^{-1}(t):=\Phi^{-1}(t)t^{-\az}$,
where $\Phi^{-1}(t):=\inf\{s\in(0,\fz):\ \Phi(s)>t\}$.
Suppose that $T_{\az}$ is as in \eqref{1.9}, with kernel $K_\az$
satisfying \eqref{1.7} and \eqref{1.8}, which is bounded from $\lp$
into $\lq$ for all $p\in (1,1/\az)$ and $1/q=1/p-\az$.
If $1<a_{\Phi}\le b_{\Phi}<\fz$ and $1<a_{\Psi}\le b_{\Psi}<\fz$, then the multilinear
commutator $T_{\az,\vec{b}}$ as in \eqref{1.11} is bounded from $\bph$ to
$\bps$, namely, there exists a positive constant $C$
such that, for all $f\in \bph$,
$$\|T_{\az,\vec{b}}f\|_{\bps}
\le C\prod_{j=1}^k\|b_j\|_{\rbmo}\|f\|_{\bph}.$$
\end{theorem}

\begin{remark}\label{r1.16}
(i) Let all the notation be the same as in Theorem \ref{t1.15}.
By Theorem \ref{t1.13}, we can, in Theorem \ref{t1.15}, replace the assumption that
$T_{\az}$ is bounded from $\lp$ into $\lq$ for all $p\in (1,1/\az)$ and $1/q=1/p-\az$
by any one of the statements (II)-(V) in Theorem \ref{t1.13}.

(ii) In Theorem \ref{t1.15}, if $p\in (1,1/\az)$ and $\Phi(t):=t^{p}$ for all
$t\in (0,\fz)$, then $\Psi(t)=t^{q}$ and $1/q=1/p-\az$. In this case,
$a_{\Phi}=b_{\Phi}=p\in (1,\fz)$, $a_{\Psi}=b_{\Psi}=q\in (1,\fz)$,
$L^{\Phi}(\mu)=L^p(\mu)$ and $L^{\Psi}(\mu)=L^q(\mu)$. Thus,
Theorem \ref{t1.15}, even when $\cx:=\rr^d$, $d$ being the usual Euclidean metric and $\mu$
as in \eqref{1.1}, also contains \cite[Theorem 1.1]{hmy1} as a special case.
In the non-homogenous setting, Theorem \ref{t1.15}, even when $k=1$, is also new.
\end{remark}

The end point counterpart of Theorem \ref{t1.15} is also considered in this paper.
To this end, we first recall the following Orlicz type space $\osc$ of functions
(see, for example, P\'erez and Trujillo-Gonz\'alez \cite{pt} for
Euclidean spaces and \cite{hmy1} for non-doubling measures).

In what follows, let $L^1_\loc(\mu)$ be the \emph{space
of all locally $\mu$-integrable functions on $\cx$}.
For all balls $B$ and $f\in L^1_\loc(\mu)$, $m_B(f)$
denotes the \emph{mean value of $f$ on ball $B$}, namely,
\begin{equation}\label{1.13}
m_B(f):=\frac{1}{\mu(B)}\int_B f(y)\,d\mu(y).
\end{equation}

\begin{definition}\label{d1.17}
Let $r\in [1,\fz)$. A function $f\in L^1_\loc(\mu)$ is said to belong to the \emph{space}
$\osc$ if there exists a positive constant $C_1$ such that

(i) for all balls $B$,
\begin{eqnarray*}
&&\|f-m_{\wz B}(f)\|_{\ep}\\
&&\hs:=\inf\lf\{\lz\in(0,\fz):\ \frac{1}{\mu(2B)}
\int_B\exp \lf(\frac{|f-m_{\wz B}(f)|}{\lz}\r)^r\,d\mu\le 2\r\}\le C_1;
\end{eqnarray*}

(ii) for all doubling balls $Q\st R$,
$|m_Q(f)-m_R(f)|\le C_1K_{Q,R}.$

The \emph{$\osc$ norm} of $f$, $\|f\|_{\osc}$, is then defined to be
the infimum of all positive constants $C_1$ satisfying (i) and (ii).
\end{definition}

\begin{remark}\label{r1.18}
Obviously, for any $r\in [1,\fz)$, $\osc\st \rbmo$. Moreover, from
\cite[Corollary 6.3]{h10}, it follows that $\rm{Osc}_{\exp \emph{L}^1}(\mu)=\rbmo$.
\end{remark}

We recall some notation from \cite{hmy2}.
For $i\in\{1,\ldots, k\}$, the \emph{family of all finite subsets
$\sz:=\{\sz(1),\ldots,\sz(i)\}$ of $\{1,\ldots,k\}$ with
$i$ different elements} is denoted by $C_{i}^{k}$. For any $\sz\in C_{i}^{k}$, the
\emph{complementary sequence $\sz'$} is defined by
$\sz':=\{1,\ldots,k\}\setminus \sz$. For any
$\sz:=\{\sz(1),\ldots,\sz(i)\}\in C_{i}^{k}$ and
$k$-tuple $r:=(r_{1},\ldots,r_{k})$, we write
that $1/r_{\sz}:=1/r_{\sz(1)}+\cdots+1/r_{\sz(i)}$ and
$1/r_{\sz'}:=1/r-1/r_{\sz}$, where
$1/r:=1/r_{1}+\cdots+1/r_{k}$.

Now we state the third main result of this paper.
\begin{theorem}\label{t1.19}
Let $\az\in(0,1)$, $k\in \nn$, $r_i\in[1,\fz)$ and $b_i\in \osi$ for $i\in\{1,\ldots,k\}$.
Let $T_\az$ and $T_{\az,\vec b}$ be, respectively, as in \eqref{1.9} and \eqref{1.11}
with kernel $K_\az$ satisfying \eqref{1.7} and \eqref{1.8}. Suppose that $T_\az$ is
bounded from $\lp$ into $\lq$ for all $p\in (1,1/\az)$ and $1/q=1/p-\az$.
Then, there exists a positive constant $C$ such that, for all $\lz\in (0,\fz)$
and $f\in L^{\fz}_b(\mu)$,
\begin{eqnarray*}
&&\mu(\{x\in \cx\,:\,|T_{\az,\vec b}f(x)|>\lz\})\\
&&\hs\le C\lf[\Phi_{1/r}\lf(\prod_{j=1}^k\|b_j\|_{\osj}\r)\r]
\lf[\sum^k_{j=0}\sum_{\sz\in C^k_j}\Phi_{1/r_\sz}
\lf(\|\Phi_{1/r_{\sz^{'}}}(\lz^{-1}|f|)\|_{\lon}\r)\r],
\end{eqnarray*}
where $\Phi_s(t):=t\log^s(2+t)$ for all $t\in(0,\fz)$ and $s\in(0,\fz)$.
\end{theorem}

\begin{remark}\label{r1.20}
Theorem \ref{t1.19} covers \cite[Theorem 1.1]{hmy4} by
taking $\cx:=\rr^d$, $d$ being the usual Euclidean metric and $\mu$
as in \eqref{1.1}.
\end{remark}

The organization of this paper is as follows.

In Section \ref{s2}, we show Theorem \ref{t1.13} by first
establishing a new interpolation theorem (see Theorem \ref{t2.7} below),
which, when $p_0=\fz$, is just \cite[Theorem 1.1]{ly2}
and whose version on the linear operators over the non-doubling setting is just \cite[Lemma 2.3]{hmy4}.
Moreover, we prove Theorem \ref{t2.7} by borrowing some ideas from the proof
of \cite[Theorem 1.1]{ly2}, which seals some gaps existing
in the proof of \cite[Lemma 2.3]{hmy4}.
The key tool for the proof of Theorem \ref{t2.7} is
the Calder\'on-Zygmund decomposition
in the non-homogeneous setting obtained by Bui and Duong \cite{bd} (see also Lemma
\ref{l2.6} below). Again, using the Calder\'on-Zygmund decomposition (Lemma \ref{l2.6})
and the interpolation theorem (Theorem \ref{t2.7}), together with the full applications
of the geometrical properties of $K_{B,S}$ and the underlying space $(\cx,d,\mu)$,
we then complete the proof of Theorem \ref{t1.13}.

Section \ref{s3} is devoted to proving Theorems \ref{t1.15} and
\ref{t1.19}. We first prove, in Theorem \ref{t3.9} below, that,
if the generalized fractional integral
$T_\az$ ($\az\in (0,1)$) is bounded from $\lp$ into $\lq$ for some
$p\in (1,1/\az)$ and $1/q=1/p-\az$, then so is its commutator
with any $\rbmo$ function,  by borrowing some ideas of
\cite[Theorem 1]{cs}. The main new ingredient appearing
in our approach used for the proof of
Theorem \ref{t3.9} is that we introduce a quantity $\wz{K}_{B,S}^{(\az)}$, which is
a fractional variant of $\wz{K}_{B,S}$
and, in the setting of non-doubling measures,
was introduced by Chen and Sawyer in \cite[Section 1]{cs}. As the case $\wz{K}_{B,S}$,
$\wz{K}_{B,S}^{(\az)}$ also well characterizes
the geometrical properties of balls $B$ and $S$ and, moreover, it preserves all the
properties of $K_{Q,R}^{(\bz)}$ in \cite[Lemma 3]{cs}. To prove Theorem \ref{t3.9},
we also need to introduce the maximal
operator $\wz{M}^{\#,\az}$, associated with $\wz{K}_{B,S}^{(\az)}$, adapted from the
maximal operator $M^{\#,(\bz)}$ in \cite[Section 2]{cs}.
Then we complete the proof of Theorem \ref{t1.15} by the interpolation
theorem in \cite{fyy} on Orlicz spaces and borrowing some ideas from
the proof of \cite[Theorem2]{hmy2}. To obtain the weak type endpoint estimates of multilinear
commutators in Theorem \ref{t1.19}, we need to use the generalized H\"older inequality
over the non-homogeneous setting from \cite[Lemma 4.1]{fyy} and
the Calder\'on-Zygmund decomposition mentioned above.

In Section \ref{s4}, under some weak reverse doubling condition of
the dominating function
$\lz$ (see Section \ref{s4} below),
which is weaker than the
assumption introduced by Bui and Duong in \cite[Subsection 7.3]{bd}:
there exists $m\in(0,\fz)$ such that, for all
$x\in\cx$ and $a,\,r\in (0,\fz)$,
$\lz(x,ar)=a^m\lz(x,r),$
we construct a non-trivial example of generalized fractional integrals
satisfying all the assumptions of this article. The key tool is the weak growth
condition (see Remark \ref{r1.4}(iii)) introduced by Tan and Li \cite{tl},
which is equivalent to the upper doubling condition.

Finally, we make some conventions on notation.
Throughout the whole paper, $C$ stands for a {\it positive constant} which
is independent of the main parameters, but it may vary from line to
line. Moreover, we use $C_{\rho,\gz,\ldots}$ or $C_{(\rho,\gz,\ldots)}$
to denote a positive constant depending on the parameter $\rho,\,\gz,\,\ldots$.
For any ball $B$ and $f\in L^1_\loc(\mu)$,
$m_B(f)$ denotes the \emph{mean value of $f$ over $B$} as in \eqref{1.13}; the center and the
radius of $B$ are denoted, respectively, by $c_B$ and $r_B$.
If $f\le Cg$, we then write $f\ls g$; if $f\ls g\ls f$, we then write $f\sim g$.
For any subset $E$ of $\cx$, we use
$\chi_E$ to denote its {\it characteristic function}.

\section{Proof of Theorem \ref{t1.13}\label{s2}}

\hskip\parindent In this section, we prove Theorem \ref{t1.13}.
We begin with recalling some useful properties of $\dz$ in Definition
\ref{d1.9} (see, for example, \cite[Lemmas 5.1 and 5.2]{h10} and \cite[Lemma 2.2]{hyy}).

\begin{lemma}\label{l2.1}
\begin{itemize}
 \item[\rm(i)] For all balls $B\st R\st S$, $K_{B,R}\le K_{B,S}$.
\vspace{-0.25cm}
 \item[\rm(ii)] For any $\rho\in [1,\fz)$, there exists
a positive constant $C_{(\rho)}$, depending on $\rho$, such that, for
all balls $B\st S$ with
$r_S\le \rho r_B$, $K_{B,S}\le C_{(\rho)}$.
\vspace{-0.25cm}
 \item[\rm(iii)] For any $\az \in(1,\fz)$, there exists
a positive constant $C_{(\az)}$, depending on $\az$, such that, for all balls $B$,
$K_{B,\wz B^\az}\le C_{(\az)}$.
\vspace{-0.25cm}
 \item[\rm(iv)] There exists a positive constant $c$ such that, for all balls
$B\st R\st S$,
$$K_{B,S}\le K_{B,R}+cK_{R,S}.$$
In particular, if $B$ and $R$ are concentric, then $c=1$.
\vspace{-0.25cm}
 \item[\rm(v)] There exists a positive constant $\wz c$ such that, for all balls
$B\st R\st S$,
$K_{R,S}\le \wz cK_{B,S};$
moreover, if $B$ and $R$ are concentric, then $K_{R,S}\le K_{B,S}$.
\end{itemize}
\end{lemma}

Now we recall the following equivalent characterizations of $\rbmo$ established
in \cite[Proposition 2.10]{hyy}.

\begin{lemma}\label{l2.2}
Let $\rho\in (1,\fz)$ and $f\in L_\loc^1(\mu)$. The following statements are equivalent:

{\rm(a)} $f\in \rbmo$;

{\rm(b)} there exists a positive constant $C$ such
that, for all balls $B$,
\begin{equation*}
\frac{1}{\mu(\rho B)}\int_{B}\lf|f(x)-m_{\wz B}f\r|\,d\mu(x)\le C
\end{equation*}
and, for all doubling balls $B\st S$,
\begin{equation}\label{2.1}
|m_B(f)-m_{S}(f)|\le CK_{B,S}.
\end{equation}
Moreover, let $\|f\|_*$ be the infimum of all admissible constants $C$ in {\rm(b)}.
Then there exists a constant $\wz C\in [1,\fz)$
such that, for all $f\in\rbmo$,
$\|f\|_*/\wz C \le \|f\|_{\rbmo}\le \wz C\|f\|_*$.
\end{lemma}

We also need the following conclusion, which is just \cite[Corollary 3.3]{fyy}.

\begin{corollary}\label{c2.3}
If $f\in\rbmo$, then there exists a positive constant $C$ such
that, for any ball $B$, $\rho\in(1,\fz)$ and $r\in [1,\fz)$,
\begin{equation}\label{2.2}
\left\{\frac1{\mu(\rho B)}\int_{B}\lf|f(x)-m_{\wz{B}}f\r|^r
\,d\mu(x)\right\}^{1/r}\le C\|f\|_{\rbmo}.
\end{equation}
Moreover, the infimum of the positive constants $C$ satisfying both
\eqref{2.2} and \eqref{2.1} is an equivalent
$\rbmo$-norm of $f$.
\end{corollary}

The following interpolation result is from \cite[Theorem 2.2]{fyy}.
\begin{lemma}\label{l2.4}
Let $\az\in [0,1)$, $p_{i},\,q_{i}\in (0,\fz)$ satisfy $1/q_{i}=1/p_{i}-\az$ for
$i\in\{1,2\}$, $p_{1}<p_{2}$ and $T$ be a sublinear
operator of weak type $(p_{i},q_{i})$ for $i\in\{1,2\}$. Then $T$
is bounded from $\bph$ to $\bps$, where $\Phi$ and
$\Psi$ are convex Orlicz functions satisfying the following conditions:
$1<p_{1}<a_{\Phi}\le b_{\Phi}<p_{2}<\fz$, $1<q_{1}<a_{\Psi}\le b_{\Psi}<q_{2}<\fz$
and, for all $t\in (0,\infty)$, $\Psi^{-1}(t)=\Phi^{-1}(t)t^{-\az}$.
\end{lemma}

We also recall some results in \cite[Subsection 4.1]{bd}
and \cite[Corollary 3.6]{h10}.

\begin{lemma}\label{l2.5}
{\rm(i)} Let $p\in (1,\fz)$, $r\in (1,p)$
and $\rho\in (0,\fz)$.
The following maximal operators defined, respectively, by setting, for all $f\in L^1_{\loc}(\mu)$
and $x\in\cx$,
$$M_{r,\rho}f(x):=\sup_{Q\ni x}\lf[\frac{1}{\mu(\rho Q)}
\int_{Q}|f(y)|^{r}\,d\mu(y)\r]^{\frac{1}{r}},$$
$$Nf(x):=\sup_{Q\ni x,\,Q\,{\rm doubling}}\frac{1}{\mu(Q)}\int_{Q}|f(y)|\,d\mu(y)$$
and
$$M_{(\rho)}f(x):=\sup_{Q\ni x}\frac{1}{\mu(\rho Q)}\int_{Q}|f(y)|\,d\mu(y),$$
are bounded on $\lp$ and also bounded from $\lon$ into $L^{1,\fz}(\mu)$.

{\rm(ii)} For all $f\in L_{\loc}^1(\mu)$, it holds true that
$|f(x)|\le Nf(x)$ for $\mu$-almost every $x\in \cx$.
\end{lemma}

Before we prove Theorem \ref{t1.13}, we establish
a new interpolation theorem, which is adapted from \cite[Theorem 1.1]{ly2}.
To this end, we first recall the following Calder\'on-Zygmund decomposition
theorem obtained by Bui and Duong \cite[Theorem 6.3]{bd}.
Let $\gz_0$ be a fixed positive constant
satisfying that $\gz_0>\max\{C_\lz^{3\log_26}, 6^{3n}\}$, where
$C_\lz$ is as in \eqref{1.2} and $n$
as in Remark \ref{r1.2}{\rm(ii)}.

\begin{lemma}\label{l2.6}
Let $p\in[1,\,\fz)$, $f\in L^p(\mu)$ and $t\in(0,\,\fz)$
($t>\gz_0^{1/p}\|f\|_\lp/[\mu(\cx)]^{1/p}$ when $\mu(\cx)<\fz$). Then
\begin{itemize}
\vspace{-0.3cm}
\item[\rm(i)] there exists a family of finite overlapping balls $\{6B_j\}_j$ such
that $\{B_j\}_j$ is pairwise disjoint,
\begin{equation}\label{2.3}
\frac{1}{\mu\lf(6^2B_j\r)}\int_{B_j}|f(x)|^p\,d\mu(x)>
\frac {t^p}{\gz_0}\
\textrm{for all}\ j,
\end{equation}
$$\frac1{\mu(6^2\eta B_j)}\int_{\eta B_j}|f(x)|^p\,d\mu(x)
\le\frac{t^p}{\gz_0}\ \textrm{for all}\ j\
\textrm{and all}\ \eta\in(2,\,\infty)$$
and
\begin{eqnarray}\label{2.4}
 |f(x)|\le t \,\,\textrm{for}\,\,\mu\textrm{-almost every}\ x\in\cx\setminus(\cup_j 6B_j);
\end{eqnarray}

\vspace{-0.3cm}
\item[\rm(ii)] for each $j$, let $R_j$ be a
$(3\times 6^2, C_\lz^{\log_2(3\times6^2)+1})$-doubling ball of
the family $\{(3\times 6^2)^k B_j\}_{k\in\nn}$, and
$\omega_j:=\chi_{6B_j}/(\sum_k \chi_{6B_k})$.
Then there exists a family $\{\vz_j\}_j$ of
functions such that, for each $j$, $\supp(\vz_j)\subset R_j$,
$\vz_j$ has a constant sign on $R_j$,
\begin{equation}\label{2.5}
\dint_\cx \vz_j(x)\,d\mu(x)=\dint_{6B_j} f(x)\omega_j(x)\,d\mu(x)
\end{equation}
and
\begin{equation}\label{2.6}
\sum_{j}|\vz_j(x)|\le \gz t \,\,\textrm{for}\,\,\mu\textrm{-almost every\,\,} x\in\cx,
\end{equation}
where $\gz$ is a positive constant depending only on
$(\cx, \mu)$ and there exists a positive constant $C$, independent of
$f$, $t$ and $j$, such that, if $p=1$, then
\begin{equation}\label{2.7}
\|\vz_j\|_\li\mu(R_j)\le C\dint_{\cx}|f(x)\omega_j(x)|\,d\mu(x)
\end{equation}
and, if $p\in (1,\,\fz)$, then
\begin{equation}\label{2.8}
\lf\{\dint_{R_j}|\vz_j(x)|^p\,d\mu(x)\r\}^{1/p}
[\mu(R_j)]^{1/p'}\le \frac{C}{t^{p-1}}\dint_{\cx}|f(x)\omega_j(x)|^p\,d\mu(x);
\end{equation}

\vspace{-0.3cm}
\item[\rm(iii)] when $p\in(1,\fz)$, if, for any $j$, choosing $R_j$ to be the smallest
$(3\times 6^2, C_\lz^{\log_2(3\times6^2)+1})$-doubling ball of
the family $\{(3\times 6^2)^k B_j\}_{k\in\nn}$, then $h:=\sum_j(f\oz_j-\vz_j)\in\hon$
and there exists a positive constant $C$, independent of $f$ and $t$, such that
\begin{equation}\label{2.9}
\|h\|_{\hon}\le\frac{C}{t^{p-1}}\|f\|^p_{\lp}.
\end{equation}
\end{itemize}
\end{lemma}

Recall that the \emph{sharp maximal operator} $M^\#$ in \cite{bd}
is defined by setting, for all $f\in L^1_{\loc}(\mu)$ and $x\in\cx$,
$$M^{\#}f(x):=\sup_{B\ni x}\frac{1}{\mu(6B)}\int_{B}|f(x)-m_{\wz{B}}f|\,d\mu(x)
+\sup_{(Q,R)\in
\bdz_{x}}\frac{|m_{Q}f-m_{R}f|}{K_{Q,R}},$$
where $\bdz_{x}:= \{(Q,R):\ \,x\in Q\st R\ {\rm and}\
Q,\,R\,\ {\rm are\ doubling\ balls}\}$.

\begin{theorem}\label{t2.7}
Let $T$ be a bounded sublinear operator from $L^{p_{0}}(\mu)$ into
$\rbmo$ and from $H^{1}(\mu)$ into
$L^{p_{0}',\fz}(\mu)$, where $p_{0}\in (1,\fz]$ and $1/p_0+1/p_0'=1$. Then $T$ extends
to a bounded linear operator from $L^{p}(\mu)$ into $L^{q}(\mu)$, where $p\in (1,p_{0})$ and
$1/q=1/p-1/p_{0}$.
\end{theorem}

\begin{proof}
By the Marcinkiewicz interpolation theorem, it suffices to prove that
\begin{equation}\label{2.10}
\mu(\{x\in \cx:\ |Tf(x)|> t\})\ls[ t^{-1}\|f\|_{L^p{(\mu)}}]^q
\end{equation}
for all $p\in(1,p_0)$ and $1/q=1/p-1/p_0$. We consider the following two cases.

Case i) $\mu(\cx)=\fz$. Let $L_{b,0}^\fz(\mu)
:=\{f\in L_b^\fz(\mu):\ \int_\cx f(x)\,d\mu(x)=0\}$.
Then, by a standard argument, we know that $L_{b,0}^\fz(\mu)$ is dense in $L^p(\mu)$
for all $p\in(1,p_0)$. Let $r\in (0,1)$. Define $N_r(g):=\lf[N(|g|^r)\r]^{1/r}$
for all $g\in L_{\loc}^r(\mu)$. By Lemma \ref{l2.5}(ii)
and a standard density argument,
to prove \eqref{2.10}, it suffices to
prove that, for any $f\in L_{b,0}^\fz(\mu)$, $p\in(1,p_0)$ and $1/q=1/p-1/p_0$,
\begin{equation}\label{2.11}
\sup_{ t\in(0,\fz)} t^q\mu(\{x\in \cx:\ |N_r(Tf)(x)|> t\})\ls\|f\|_{\lp}^q.
\end{equation}
To this end, for any given $f\in L^\fz_{b,0}(\mu)$, applying Lemma
\ref{l2.6} to $f$ with $ t$ replaced by $ t^{q/p}$, and letting $R_j$ be as in Lemma \ref{l2.6}(iii),
we see that
$f=g+h$, where $g:=f\chi_{\cx\setminus\cup_j6B_j}+\sum_j\vz_j$ and
$h:=\sum_j(\oz_jf-\vz_j)$.
By Minkowski's inequality, H\"older's inequality and $1/q=1/p-1/p_0$, together with
\eqref{2.4}, \eqref{2.6} and \eqref{2.8} with $t$ replaced by $t^{q/p}$, we conclude that
\begin{eqnarray}\label{2.12}
\|g\|_{L^{p_0}(\mu)}&&\le\lf\|f\chi_{\cx\setminus\cup_j6B_j}\r\|_{L^{p_0}(\mu)}
+\lf\|\sum_j\vz_j\r\|_{L^{p_0}(\mu)}\\
&&\noz\ls t^{q(\frac1p-\frac1{p_0})}\|f\|_{\lp}^{p/p_0}+ t^{(q/p)/p'_0}
\lf[\sum_j\|\vz_j\|_{\lon}\r]^{1/p_0}\\
&&\noz\ls t\|f\|_{\lp}^{p/p_0}+
 t^{(q/p)/p'_0}\lf[\sum_j\|\vz_j\|_\lp[\mu(R_j)]^{1/p'}\r]^{1/p_0}\\
&&\ls t\|f\|_{\lp}^{p/p_0}+ t^{(q/p)/p'_0} t^{-q/(p'p_0)}
\lf[\sum_j\int_\cx|\oz_j(x)f(x)|^p\,d\mu(x)\r]^{1/p_0}\noz\\
&&\ls t\|f\|_{\lp}^{p/p_0}.\noz
\end{eqnarray}
For each $r\in (0,1)$, define $M_r^\#g:=\lf\{M^\#(|g|^r)\r\}^{1/r}$.
Then, from \cite[Lemma 3.1]{ly2}, together with the boundedness of $T$ from $L^{p_0}(\mu)$
into $\rbmo$ and \eqref{2.12}, we deduce that
$$\|M_r^\#Tg\|_{\li}\ls\|Tg\|_{\rbmo}\ls\|g\|_{L^{p_0}(\mu)}\ls t\|f\|_{\lp}^{p/p_0}.$$
Hence, if $C_0$ is chosen to be a sufficiently large positive constant, we then see that
\begin{equation}\label{2.13}
\mu\lf(\lf\{x\in\cx:\ M_r^\#(Tg)(x)>C_0 t\|f\|_{\lp}^{p/p_0}\r\}\r)=0.
\end{equation}
On the other hand, since both $f$ and $h$ belong to $\hon$, by \eqref{2.9} with $t$
replaced by $t^{q/p}$, we conclude that $g\in \hon$ and
$$\|g\|_{\hon}\le\|f\|_{\hon}+\|h\|_{\hon}\ls\|f\|_{\hon}
+\frac1{ t^{(p-1)q/p}}\|f\|_{\lp}^p.$$
From this, together with the boundedness of $T$ from $\hon$ into $L^{p'_0,\fz}(\mu)$ and
\cite[Lemma 3.3]{ly2}, we deduce that, for any $q$ satisfying $1/q=1/p-1/p_0$ and
$R\in(0,\fz)$,
\begin{eqnarray}\label{2.14}
&&\sup_{ t\in(0,R)} t^q\mu\lf(\{x\in\cx:\ N_r(Tg)(x)> t\}\r)\\
&&\hs\ls\sup_{ t\in(0,R)} t^{q-p'_0}\sup_{\tau\in[ t,\fz)}\tau^{p'_0}
\mu\lf(\{x\in\cx:\ |Tg(x)|>\tau\}\r)\noz\\
&&\hs\ls R^{q-p_0}\|Tg\|_{L^{p'_0,\fz}(\mu)}\ls R^{q-p_0}\|g\|_{\hon}<\fz.\noz
\end{eqnarray}
From the fact that $N_r\circ T$ is quasi-linear, \eqref{2.14}, \cite[Lemma 3.2]{ly2}
and \eqref{2.13}, we deduce that there exists a positive constant $\wz C$ such that,
for all $f\in L^\fz_{b,0}(\mu)$,
\begin{eqnarray}\label{2.15}
&&\sup_{ t\in(0,\fz)} t^q
\mu\lf(\lf\{x\in\cx:\ N_r(Tf)(x)>\wz{C}C_0 t\|f\|_{\lp}^{p/p_0}\r\}\r)\\
&&\noz\hs\ls\sup_{ t\in(0,\fz)} t^q
\mu\lf(\lf\{x\in\cx:\ N_r(Tg)(x)>C_0 t\|f\|_{\lp}^{p/p_0}\r\}\r)\\
&&\noz\hs\hs+\sup_{ t\in(0,\fz)} t^q
\mu\lf(\lf\{x\in\cx:\ N_r(Th)(x)>C_0 t\|f\|_{\lp}^{p/p_0}\r\}\r)\\
&&\noz\hs\ls\sup_{ t\in(0,\fz)} t^q
\mu\lf(\lf\{x\in\cx:\ M_r^\#(Tg)(x)>C_0 t\|f\|_{\lp}^{p/p_0}\r\}\r)\\
&&\noz\hs\hs+\sup_{ t\in(0,\fz)} t^q
\mu\lf(\lf\{x\in\cx:\ N_r(Th)(x)>C_0 t\|f\|_{\lp}^{p/p_0}\r\}\r)\\
&&\noz\hs\sim\sup_{ t\in(0,\fz)} t^q
\mu\lf(\lf\{x\in\cx:\ N_r(Th)(x)> t\|f\|_{\lp}^{p/p_0}\r\}\r).
\end{eqnarray}
By the boundedness of $N$ from $\lon$ into $L^{1,\fz}(\mu)$ (see Lemma \ref{l2.5}(i)),
the layer cake representation, the boundedness of
$T$ from $\hon$ into $L^{p'_0,\fz}(\mu)$ and \eqref{2.9}
with $t$ replaced by $t^{q/p}$, we conclude that
\begin{eqnarray}\label{2.16}
&&\mu\lf(\lf\{x\in\cx:\ N_r(Th)(x)> t\|f\|_{\lp}^{p/p_0}\r\}\r)\\
&&\noz\hs=\mu\lf(\lf\{x\in\cx:\ N(|Th|^r)(x)> t^r\|f\|_{\lp}^{rp/p_0}\r\}\r)\\
&&\noz\hs\le\mu\lf(\lf\{x\in\cx:\ N(|Th|^r\chi_{\{y\in\cx:\
|Th(y)|>2^{-1/r} t\|f\|_{\lp}^{p/p_0}\}})(x)>\frac{ t^r}{2}
\|f\|_{\lp}^{rp/p_0}\r\}\r)\\
&&\noz\hs\ls t^{-r}\|f\|_\lp^{-rp/p_0}\int_\cx|Th(x)|^r\chi_{\lf\{x\in\cx:\
|Th(x)|>2^{-1/r} t\|f\|_{\lp}^{p/p_0}\r\}}(x)\,d\mu(x)\\
&&\noz\hs\sim  t^{-r}\|f\|_\lp^{-rp/p_0}\lf[\int_0^{2^{-1/r} t\|f\|_\lp^{p/p_0}}
s^{r-1}\r.\\
&&\noz\hs\hs\times\mu\lf(\lf\{x\in\cx:\ |Th(x)|>2^{-1/r} t\|f\|_{\lp}^{p/p_0}\r\}\r)\,ds\\
&&\noz\hs\hs\lf.+\int_{2^{-1/r} t\|f\|_\lp^{p/p_0}}^\fz
s^{r-1}\mu\lf(\lf\{x\in\cx:\ |Th(x)|>s\r\}\r)\,ds\r]\\
&&\noz\hs\ls\mu\lf(\lf\{x\in\cx:\ |Th(x)|>2^{-1/r} t\|f\|_{\lp}^{p/p_0}\r\}\r)\\
&&\noz\hs\hs+\lf[ t\|f\|_{\lp}^{p/p_0}\r]^{-p'_0}\sup_{s\in(0,\fz)}s^{p'_0}
\mu\lf(\lf\{x\in\cx:\ |Th(x)|>s\r\}\r)\\
&&\noz\hs
\ls \|h\|_{\hon}^{p_0'}\lf[ t\|f\|_{\lp}^{p/p_0}\r]^{-p'_0}\ls t^{-q}
\|f\|_{\lp}^p,
\end{eqnarray}
which, together with \eqref{2.15}, completes the proof of \eqref{2.11}.

Case ii) $\mu(\cx)<\fz$. In this case, assume that $f\in L_b^\fz(\mu)$. Notice that, if $ t\in(0, t_0]$, where $ t_0^q:=\bz_6\|f\|_{\lp}^q/\mu(\cx)$, then \eqref{2.10}
holds true trivially. Thus, we only need to consider the case when $ t\in( t_0,\fz)$.
Let $N_r$ and $M_r$ be as in Case i). For each $ t\in( t_0,\fz)$, applying Lemma \ref{l2.6} to $f$
with $ t$ replaced by $ t^{q/p}$, we then see that
$f=g+h$ with $g$ and $h$ as in Case i), which, together
with the boundedness of $T$ from $L^{p_0}(\mu)$ into $\rbmo$ and \cite[Lemma 3.1]{ly2},
shows that \eqref{2.13} still holds true for $M_r^\#(Tg)$.

We now claim that, for any $r\in(0,1)$,
\begin{equation}\label{2.17}
F:=\frac1{\mu(\cx)}\int_\cx|Tg(x)|^r\,d\mu(x)\ls t^r\|f\|_{\lp}^{rp/p_0},
\end{equation}
where the implicit positive constant only depends on $\mu(\cx)$ and $r$. To see this,
since $\mu(\cx)<\fz$, we may regard $\cx$ as a ball, then $g_0:=g-\frac1{\mu(\cx)}
\int_\cx g(x)\,d\mu(x)\in\hon$. Precisely, by \eqref{2.12}, we see that
\begin{equation}\label{2.18}
\|g_0\|_{\hon}\ls t\|f\|_{\lp}^{p/p_0}.
\end{equation}
On the other hand, by H\"older's inequality, the fact that $T1\in\rbmo$ and
the locally integrability of $\rbmo$ functions, we conclude that
$$\int_\cx|T1(x)|^r\,d\mu(x)\le\lf[\int_\cx|T1(x)|\,d\mu(x)\r]^r[\mu(\cx)]^{1-r}<\fz.$$
From this and the layer cake representation, together with $r\in (0,1)$,
H\"older's inequality, \eqref{2.12}, the boundedness of $T$
from $\hon$ into $L^{p'_0,\fz}(\mu)$ and \eqref{2.18}, we deduce that
\begin{eqnarray*}
&&\int_\cx|Tg(x)|^r\,d\mu(x)\\
&&\hs\le\int_\cx\lf\{|Tg_0(x)|^r
+\lf|\frac1{\mu(\cx)}\int_\cx g(y)\,d\mu(y)\r|^r|T1(x)|^r\r\}\,d\mu(x)\\
&&\hs\ls \int_0^{\|g_0\|_{\hon}/\mu(\cx)}t^{r-1}\mu(\{x\in\cx:\ |Tg_0(x)|>t\})\,dt
+\int_{\|g_0\|_{\hon}/\mu(\cx)}^\fz\cdots
+\|g\|_{L^{p_0}(\mu)}^r\\
&&\hs\ls\int_0^{\|g_0\|_{\hon}/\mu(\cx)}t^{r-1}\,dt
+\|g_0\|_{\hon}^{p'_0}\int_{\|g_0\|_{\hon}/\mu(\cx)}^\fz t^{r-1-p'_0}\,dt
+ t^r\|f\|_{\lp}^{rp/p_0}\\
&&\hs\ls\|g_0\|^r_{\hon}
+ t^r\|f\|_{\lp}^{rp/p_0}\ls t^r\|f\|_{\lp}^{rp/p_0},
\end{eqnarray*}
which implies \eqref{2.17}.

Observe that $\int_\cx[|Tg(x)|^r-F]\,d\mu(x)=0$ and, for any $R\in(0,\fz)$,
$$\sup_{ t\in(0,R)} t^q\mu(\{x\in\cx:\ N(|Tg|^r-F)(x)> t\})\le R^q\mu(\cx)<\fz.$$
From this and \eqref{2.17}, together with \cite[Lemma 3.2]{ly2},
$M_r^\#(F)=0$, \eqref{2.13} and some arguments similar to those used in the
 estimates for \eqref{2.15} and \eqref{2.16}, we deduce
that there exists a positive constant $\wz c$ such that
\begin{eqnarray*}
&&\sup_{ t\in( t_0,\fz)} t^q
\mu\lf(\lf\{x\in\cx:\ N_r(Tf)(x)>\wz{c}C_0 t\|f\|_{\lp}^{p/p_0}\r\}\r)\\
&&\hs\ls\sup_{ t\in( t_0,\fz)} t^q
\mu\lf(\lf\{x\in\cx:\ N(|Tg|^r-F)(x)>(C_0 t)^r\|f\|_{\lp}^{rp/p_0}\r\}\r)\\
&&\hs\hs+\sup_{ t\in( t_0,\fz)} t^q
\mu\lf(\lf\{x\in\cx:\ N_r(Th)(x)>C_0 t\|f\|_{\lp}^{p/p_0}\r\}\r)\\
&&\hs\ls\sup_{ t\in(0,\fz)} t^q
\mu\lf(\lf\{x\in\cx:\ M_r^\#(Tg)(x)>C_0 t\|f\|_{\lp}^{p/p_0}\r\}\r)\\
&&\hs\hs+\sup_{ t\in(0,\fz)} t^q
\mu\lf(\lf\{x\in\cx:\ N_r(Th)(x)>C_0 t\|f\|_{\lp}^{p/p_0}\r\}\r)\\
&&\hs\sim\sup_{ t\in(0,\fz)} t^q
\mu\lf(\lf\{x\in\cx:\ N_r(Th)(x)> t\|f\|_{\lp}^{p/p_0}\r\}\r)\ls t^{-q}
\|f\|_{\lp}^p,
\end{eqnarray*}
where $C_0$ is chosen to be a sufficiently large positive constant, which completes the proof of Theorem \ref{t2.7}.
\end{proof}

\begin{proof}[Proof of Theorem \ref{t1.13}]
(I)$\Rightarrow$(II) Let $f\in L^1(\mu)$. Without loss of generality, we may assume that $\|f\|_{L^1(\mu)}=1$.
We denote $1/(1-\az)$ by $q_0$.  Applying
Lemma \ref{l2.6} to $f$ with $p=1$ and $ t$ replaced by $ t^{q_0}$, and letting $R_j$ be as in Lemma \ref{l2.6}(iii),
we see that
$f=g+h$, where $g:=f\chi_{\cx\setminus(\cup_j6B_j)}+\sum_j\vz_j$ and
$h:=\sum_j(\oz_jf-\vz_j)$.
By \eqref{2.7} and the assumption $\|f\|_{L^1(\mu)}=1$, we easily see that
\begin{eqnarray}\label{2.19}
\|g\|_{\lon}\ls\|f\|_{\lon}\sim 1.
\end{eqnarray}
From \eqref{2.4} and \eqref{2.6} with $t$ replaced by $t^{q_0}$, it follows that, for $\mu$-almost every $x\in \cx$,
\begin{equation}\label{2.20}
|g(x)|\ls  t^{q_0}.
\end{equation}
Since $T_{\az}$ is bounded from $\lpo$ into $\lqo$ for any $p_1\in (1,1/\az)$
and $1/q_1=1/p_1-\az$, by \eqref{2.20} and \eqref{2.19}, we conclude that
\begin{eqnarray}\label{2.21}
\mu(\{x\in \cx:|T_{\az}g(x)|> t\})&&\ls t^{-q_1}\|T_\az g\|_{\lqo}^{q_1}
\ls t^{-q_1}\|g\|_{\lpo}^{q_1}\\
&&\noz\ls t^{-q_1}( t^{q_0})^{(p_1-1)q_1/p_1}\ls t^{-q_0}.
\end{eqnarray}

On the other hand, from \eqref{2.3} with $p=1$ and $t$ replaced by $t^{q_0}$,
and the fact that $\{B_j\}_j$ is a sequence of pairwise disjoint balls, we deduce that
\begin{equation}\label{2.22}
\mu(\cup_j6^2B_j)\ls  t^{-q_0}\int_\cx|f(y)|\,d\mu(y)
\ls  t^{-q_0}.
\end{equation}
Therefore, to show (II), by $f=g+h$, \eqref{2.21} and \eqref{2.22}, it suffices
to prove that
\begin{equation}\label{2.23}
\mu\lf(\lf\{x\in \cx\setminus (\cup_j6^2B_j):\ |T_{\az}h(x)|> t\r\}\r)\ls  t^{-q_0}.
\end{equation}
To this end, denote the center of $B_j$ by $x_j$, and let $N_1$ be the positive integer
satisfying $R_j=(3\times 6^2)^{N_1}B_j$.
Let $\tz$ be a bounded function with $\|\tz\|_{L^{q_0'}(\mu)}\le 1$ whose support
is contained in $\cx\setminus(\cup_j6^2B_j)$. Then
\begin{eqnarray*}
\int_{\cx\setminus(\cup_j6^2B_j)}|T_{\az}h(x)\tz(x)|\,d\mu(x)&&
\le\sum_j\int_{\cx\setminus 6R_j}|T_{\az}h_j(x)\tz(x)|\,d\mu(x)
+\sum_j\int_{6R_j\setminus 6^2B_j}\cdots\\
&&=:\mathrm{F_1}+\mathrm{F_2},
\end{eqnarray*}
where $h_j:=\oz_jf-\vz_j$. By \eqref{2.5}, we see that
$\int_\cx h_j(x)\,d\mu(x)=0$, which, together with \eqref{1.8}, H\"older's inequality
and \eqref{2.7}, further implies that
\begin{eqnarray*}
\mathrm{F_1}&&\le \sum_j\int_{\cx\setminus 6R_j}\int_\cx|\tz(x)||K_{\az}(x,y)-K_{\az}(x,x_j)|
|h_j(y)|\,d\mu(y)\,d\mu(x)\\
&&\ls \sum_j\int_\cx\lf[\sum_{i=1}^{\fz}
\int_{6^{i+1}B_j\setminus6^iB_j}\frac{r_{B_j}^{\dz}}
{(6^ir_{B_j})^{\dz}[\lz(x_j,6^ir_{B_j})]^{1-\az}}|\tz(x)|\,d\mu(x)\r]
|h_j(y)|\,d\mu(y)\\
&&\ls \sum_j\int_\cx|f(y)\oz_j(y)|\,d\mu(y)\sum_{i=1}^{\fz}
6^{-i\dz}\|\tz\|_{L^{q_0'}(\mu)}\ls 1.
\end{eqnarray*}
For ${\rm F}_2$, by $h_j:=\oz_jf-\vz_j$, \eqref{1.7}, H\"older's inequality
and an argument similar to that used in the proof of
\cite[Lemma 3.5(iii)]{fyy}, together with the boundedness
of $T_\az$ from $L^{p_2}(\mu)$ into $L^{q_2}(\mu)$ with $p_2\in(1,1/\az)$ and
$1/q_2=1/p_2-\az$, we have
\begin{eqnarray*}
\mathrm{F_2}
&&\le \sum_j\int_{6R_j\setminus 6^2B_j}|\tz(x)||T_{\az}(\oz_jf)(x)|\,d\mu(x)
+\sum_j\int_{6R_j}|\tz(x)||T_{\az}\vz_j(x)|\,d\mu(x)\\
&&\ls \sum_j\int_{6R_j\setminus6^2B_j}\frac{|\tz(x)|}{[\lz(x_j,d(x,x_j))]^{1-\az}}\,d\mu(x)
\int_\cx|f(y)\oz_j(y)|\,d\mu(y)\\
&&\hs+\sum_j\lf[\int_{6R_j}|T_{\az}\vz_j(x)|^{q_0}\,d\mu(x)\r]^{1/q_0}
\|\tz\|_{L^{q_0'}(\mu)}\\
&&\ls \sum_j\int_\cx|f(y)\oz_j(y)|\,d\mu(y)
\lf[\sum_{k=1}^{N_1+1}\frac{\mu((3\times6^2)^kB_j)}
{\lz(x_j,(3\times6^2)^kr_{B_j})}\r]^{1/q_0}\|\tz\|_{L^{q_0'}(\mu)}\\
&&\hs\hs
+\sum_j\lf[\int_{6R_j}|T_{\az}\vz_j(x)|^{q_2}\,d\mu(x)\r]^{1/q_2}
\lf[\mu(6R_j)\r]^{1/q_0-1/q_2}\ls 1,
\end{eqnarray*}
where we chose $p_2$ and $q_2$ such that $p_2\in (1,1/\az)$ and $1/q_2=1/p_2-\az$.
The estimates for $\mathrm{F_1}$ and $\mathrm{F_2}$ give \eqref{2.23}, and hence
complete the proof of (I)$\Rightarrow$(II).

(II)$\Rightarrow$(III) Indeed, for any $f\in L^{1/\az}(\mu)$, to show $T_\az f\in\rbmo$,
by the assumption that $T_\az f$ is finite almost everywhere, it suffices to show that,
for any ball $Q$ and
$h_{Q}=m_{Q}(T_{\az}(f\chi_{\cx\setminus (6/5)Q}))$,
\begin{equation}\label{2.24}
\frac{1}{\mu(6Q)}\int_{Q}|T_{\az}f(x)-h_{Q}|\,d\mu(x) \ls
\|f\|_{L^{1/\az}(\mu)}
\end{equation}
and, for any two balls $Q\st R$, where $R$ is doubling,
\begin{equation}\label{2.25}
 |h_{Q}-h_{R}|\ls K_{Q,R}\|f\|_{L^{1/\az}(\mu)}.
\end{equation}

Now we first show \eqref{2.24}. Write
\begin{eqnarray*}
\frac{1}{\mu(6Q)}\int_{Q}|T_{\az}f(x)-h_{Q}|\,d\mu(x)&& \le
\frac{1}{\mu(6Q)}\int_{Q}|T_{\az}(f\chi_{(6/5)Q})(x)|\,d\mu(x)\\
&&\hs+\frac{1}{\mu(6Q)}\int_{Q}|T_{\az}(f\chi_{\cx\setminus
(6/5)Q})(x)-h_{Q}|\,d\mu(x)
=:\mathrm{H}+\mathrm{I}.
\end{eqnarray*}

Notice that Kolmogorov's inequality (see, for example, \cite[p.\,485,\ Lemma 2.8]{gr})
also holds true in the non-homogeneous setting. By Kolmogorov's inequality, namely, for
$0<p<q<\fz$ and any function $f$,
$$\|f\|_{L^{q,\fz}(\mu)}
\le\sup_E\|f\chi_E\|_{\lp}/\|\chi_E\|_{L^s(\mu)}\ls\|f\|_{L^{q,\fz}(\mu)},$$
where $1/s=1/p-1/q$ and the supremum is taken over all measurable sets $E$ with
$0<\mu(E)<\fz$, together with (II) of Theorem \ref{t1.13} and H\"older's inequality,
we know that
\begin{equation*}
\mathrm{H}\ls
\frac{1}{\mu(6Q)}\|\chi_{Q}\|_{L^{1/\az}(\mu)}
\|T_{\az}(f\chi_{(6/5)Q})\|_{L^{q_0,\fz}(\mu)}
\ls\frac{[\mu(Q)]^\az}{\mu(6Q)}\|f\chi_{(6/5)Q}\|_{\lon}
\ls\|f\|_{L^{1/\az}(\mu)}.
\end{equation*}
To estimate I, we write
\begin{eqnarray*}
&&\lf|T_{\az}(f\chi_{\cx\setminus
(6/5)Q})(x)-T_{\az}(f\chi_{\cx\setminus (6/5)Q})(y)\r|\\
&&\hs\le\int_{6Q\setminus (6/5)Q}|K_{\az}(x,z)-K_{\az}(y,z)||f(z)|\,d\mu(z)\\
&&\hs=\int_{\cx\setminus 6Q}|K_{\az}(x,z)-K_{\az}(y,z)||f(z)|\,d\mu(z)
+\int_{\cx\setminus (6/5)Q}\cdots=:{\rm I}_1+{\rm I}_2.
\end{eqnarray*}
Let $c_Q$ and $r_Q$ be the center and the radius of $Q$, respectively.
To estimate ${\rm I}_1$, from \eqref{1.7} and H\"older's inequality, together with
\eqref{1.2} and \eqref{1.3}, it follows that
\begin{eqnarray*}
{\rm I}_1&&\ls\int_{6Q\setminus (6/5)Q}\lf(\frac1{[\lz(x,d(x,z))]^{1-\az}}
+\frac1{[\lz(y,d(y,z))]^{1-\az}}\r)|f(z)|\,d\mu(z)\\
&&\ls\frac1{[\lz(c_Q, r_Q)]^{1-\az}}\int_{6Q}|f(z)|\,d\mu(z)\ls\|f\|_{L^{1/\az}(\mu)}.
\end{eqnarray*}

To estimate $\mathrm{I}_2$, by \eqref{1.8}, \eqref{1.2}, H\"older's inequality and \eqref{1.3},
we see that, for any $x,\,y\in Q$,
\begin{eqnarray*}
{\rm I}_2
&&\ls \sum_{i=1}^{\fz}\int_{2^{i}(6Q)\setminus
2^{i-1}(6Q)}\frac{[d(x,y)]^{\dz}}{[d(z,y)]^{\dz}[\lz(y,d(z,y))]^{1-\az}}
|f(z)|\,d\mu(z)\\
&&\ls\sum_{i=1}^{\fz}\int_{2^{i}(6Q)\setminus
2^{i-1}(6Q)}\frac{r_Q^{\dz}}{[2^{(i-1)}(6r_Q)]^{\dz}¡¢
[\lz(y,2^{(i-1)}6r_Q)]^{1-\az}}
|f(z)|\,d\mu(z)\\
&&\ls\sum_{i=1}^{\fz}2^{-(i-1)\dz}
\lf[\frac{\mu(2^i (6Q))}{\lz(c_Q,2^i(6r_Q))}\r]^{1-\az}\|f\|_{L^{1/\az}(\mu)}
\ls\|f\|_{L^{1/\az}(\mu)}.
\end{eqnarray*}
Therefore, $\mathrm{I}\ls \|f\|_{L^{1/\az}(\mu)}$.

Combining the estimates for $\mathrm{H}$ and $\mathrm{I}$, we obtain \eqref{2.24}.

Now we show \eqref{2.25} for the chosen $\{h_{Q}\}_{Q}$. Denote
$N_{Q,R}+1$ simply by $N_2$. Write
\begin{eqnarray*}
|h_{Q}-h_{R}|&&=|m_{Q}(T_{\az}(f\chi_{\cx\setminus
(6/5)Q}))-m_{R}(T_{\az}(f\chi_{\cx\setminus (6/5)R}))|\\
&&\le |m_{Q}(T_{\az}(f\chi_{6Q\setminus
(6/5)Q}))|+|m_{Q}(T_{\az}(f\chi_{6^{N_2}Q\setminus 6Q}))|\\
&&\hs+|m_{Q}(T_{\az}(f\chi_{\cx\setminus
6^{N_2}Q}))-m_{R}(T_{\az}(f\chi_{\cx\setminus
6^{N_2}Q}))|+|m_{R}(T_{\az}(f\chi_{6^{N_2}Q\setminus (6/5)R}))|\\
&&=:\mathrm{J_{1}}+\mathrm{J_{2}}+\mathrm{J_{3}}+\mathrm{J_{4}}.
\end{eqnarray*}

An argument similar to that used in the estimate for
$\mathrm{H}$ shows that $\mathrm{J_{4}}\ls
\|f\|_{L^{1/\az}(\mu)}$. Also, an argument similar to
that used in the estimate for
$\mathrm{I}$ gives us that $\mathrm{J_{3}}\ls
\|f\|_{L^{1/\az}(\mu)}$.

Next we estimate $\mathrm{J_{2}}$.
For any $x\in Q$, by H\"older's inequality, the fact that $6^{N_2}Q\st72R$ and
(ii) and (iv) of Lemma \ref{l2.1}, we have
\begin{eqnarray*}
\lf|T_{\az}(f\chi_{6^{N_2}Q\setminus 6Q})(x)\r|&&\le \lf[
\int_{6^{N_2}Q\setminus 6Q}\frac{1}{\lz(x,d(x,z))}\,d\mu(z)\r]^{1-\az}
\|f\|_{L^{1/\az}(\mu)}\\
&&\ls K_{Q,36R}\|f\|_{L^{1/\az}(\mu)}\ls K_{Q,R}\|f\|_{L^{1/\az}(\mu)}.
\end{eqnarray*}
This implies that $\mathrm{J_{2}}\ls
K_{Q,R}\|f\|_{L^{1/\az}(\mu)}$. Similarly, we have
$${\rm J_1}\ls K_{Q,6Q}\|f\|_{L^{1/\az}(\mu)}\ls K_{Q,R}\|f\|_{L^{1/\az}(\mu)}.$$

Combining the estimates for
$\mathrm{J_{1}}$, $\mathrm{J_{2}}$, $\mathrm{J_{3}}$ and
$\mathrm{J_{4}}$, we obtain \eqref{2.25} and hence complete the
proof of (II)$\Rightarrow$(III).

(III)$\Rightarrow$(IV) We first show that, for any ball $B$, bounded
function $a$
supported on $B$ and $q_0:=1/(1-\az)$,
\begin{equation}\label{2.26}
\int_B|T_\az a(x)|^{q_0}\,d\mu(x)\ls[\mu(2B)]^{q_0}\|a\|_{\li}^{q_0}.
\end{equation}

To prove this, we borrow some ideas from the proof of \cite[Lemma 3.1]{lyy} by
considering the following two cases for $r_B$.

Case (i) $r_B\le {\rm diam}(\supp\mu)/40$, where ${\rm diam}(\supp\mu)$ denotes
the \emph{diameter of the set $\supp\mu$}.
By Corollary \ref{c2.3} and (III) of Theorem \ref{t1.13}, we have
\begin{equation}\label{2.27}
\int_B|T_\az a(x)-m_{\wz B}(T_\az a)|^{q_0}\,d\mu(x)
\ls\mu(2B)\|a\|_{L^{1/\az}(\mu)}^{q_0}
\ls[\mu(2B)]^{q_0}\|a\|_{\li}^{q_0}.
\end{equation}
Thus, by \eqref{2.27}, to prove \eqref{2.26}, it suffices to show that
\begin{equation}\label{2.28}
|m_{\wz B}(T_\az a)|\ls[\mu(2B)]^\az\|a\|_{\li}.
\end{equation}

We first claim that there exists $j_0\in\nn$
such that
\begin{equation}\label{2.29}
\mu(6^{j_0}B\setminus 2B)>0.
\end{equation}
Indeed, if, for all $j\in\nn$, $\mu(6^jB\setminus 2B)=0$, then we see that $\mu(\cx\setminus 2B)=0$,
which implies that $\supp\mu\st\overline{2B}$, the closure of $2B$.
This contradicts to that $r_B\le{\rm diam}(\supp\mu)/40$ and thus \eqref{2.29} holds true.
Now assume that $S$ is the smallest ball of the form $6^jB$ such that $\mu(S\setminus 2B)>0$.
We then know that $\mu(6^{-1}S\setminus 2B)=0$ and $\mu(S\setminus 2B)>0$.
Thus, $\mu(S\setminus(6^{-1}S\cup 2B))>0$.
By this and \cite[Lemma 3.3]{h10}, we choose $x_0\in S\setminus(6^{-1}S\cup 2B)$ such that
the ball centered at $x_0$ with the radius $6^{-k}r_S$ for some $k\ge 2$ is doubling.
Let $B_0$ be the biggest ball of this form. Then we see that $B_0\st 2S$ and
${\rm dist}(B_0,B)\gs r_B$. We now claim that
\begin{equation}\label{2.30}
K_{B,2S}\ls 1.
\end{equation}
Indeed, if $S=6B$, then by Lemma \ref{l2.1}(ii), we have \eqref{2.30}.
If $S\supset 6^2B$, then $(1/12)S\supset3B$. Notice that, in this case,
$\mu(6^{-1}S\setminus2B)=0$ implies that $K_{2B,(1/12)S}=1$. By this, together with
(iv) and (ii) of Lemma \ref{l2.1}, we further have
\begin{equation*}
K_{B,2S}\ls K_{B,2B}+K_{2B,(1/12)S}+K_{(1/12)S,2S}
\ls K_{B,2B}+K_{(1/12)S,2S}\ls1.
\end{equation*}
Thus, \eqref{2.30} also holds true in this case, which shows \eqref{2.30}.
Moreover, assume that $r_{B_0}=6^{-k_0}r_S$, where $k_0\ge 2$, and there exists
$N\in\nn$ such that $\wz{6B_0}=6^{N+1}B_0$. By the definition of $B_0$, we
see that $N-k_0+1\ge-1$, hence $r_{6(\wz{6B_0})}\ge r_S$ and
$2S\st24(\wz{6B_0})$. Therefore, by (i) through (iv) of Lemma \ref{l2.1},
we see that
\begin{equation}\label{2.31}
K_{B_0,2S}\le K_{B_0,24(\wz{6B_0})}\ls K_{B_0,\wz{6B_0}}
+K_{\wz{6B_0},24(\wz{6B_0})}\ls1.
\end{equation}
By \eqref{2.1}, \eqref{2.31}, \eqref{2.30}, Lemma \ref{l2.1}(iii) and Theorem \ref{t1.13}(III), we know that
\begin{eqnarray}\label{2.32}
\qquad&&|m_{B_0}(T_\az a)-m_{\wz B}(T_\az a)|\\
&&\hs\le|m_{B_0}(T_\az a)-m_{2S}(T_\az a)|+|m_{2S}(T_\az a)-m_B(T_\az a)|
+|m_B(T_\az a)-m_{\wz B}(T_\az a)|\noz\\
&&\hs\ls (K_{B_0,2S}+K_{B,2S}+K_{B,\wz{B}})\|T_\az a\|_{\rbmo}\noz\\
&&\hs\ls\|a\|_{L^{1/\az}(\mu)}\ls[\mu(2B)]^\az\|a\|_{\li}\noz,
\end{eqnarray}
Moreover, by \eqref{1.7}, ${\rm dist}(B_0,B)\gs r_B$, \eqref{1.2} and \eqref{1.3}, we conclude that, for all $y\in B_0$,
\begin{equation}\label{2.33}
|T_\az a(y)|\ls\frac{\mu(B)}{[\lz(c_B,r_B)]^{1-\az}}\|a\|_{\li}
\ls[\mu(2B)]^\az\|a\|_{\li}.
\end{equation}
The estimate \eqref{2.28} follows from \eqref{2.32} and \eqref{2.33},
which completes the proof of \eqref{2.26} in this case.

Case (ii) $r_B> {\rm diam}(\supp\mu)/40$. In this case, without loss of
generality, we may assume that $r_B\le8{\rm diam}(\supp\mu)$. Then, by
Remark \ref{r1.2}(ii), we see that $B\cap\supp\mu$ is covered by finite number balls
$\{B_j\}_{j=1}^J$ with radius $r_B/800$, where $J\in\nn$ is independent of $r_B$.
For any $j\in\{1,\ldots,J\}$, we define
$a_j:=\frac{\chi_{B_j}}{\sum_{k=1}^J\chi_{B_k}}a$. Since \eqref{2.26} is
true if we replace $B$ by $2B_j$ which contains the support of $a_j$, by \eqref{1.7},
\eqref{2.26}, \eqref{1.3}, \eqref{1.2} and the fact that,
if $B\cap B_j\neq \emptyset$, then $4B_j\st 2B$, we have
\begin{eqnarray*}
&&\int_B|T_\az a(x)|^{q_0}\,d\mu(x)\\
&&\hs\ls\sum_{j=1}^J\int_{B\setminus 2B_j}|T_\az a(x)|^{q_0}\,d\mu(x)
+\sum_{j=1}^J\int_{2B_j}\cdots\\
&&\hs\ls\sum_{j=1}^J\int_{B\setminus 2B_j}\lf[\int_{B_j}
\frac{|a_j(y)|}{[\lz(x,d(x,y))]^{1-\az}}\,d\mu(y)\r]^{q_0}\,d\mu(x)
+\sum_{j=1}^J\|a_j\|_{\li}^{q_0}[\mu(4B_j)]^{q_0}\\
&&\hs\ls\sum_{j=1}^J\|a_j\|_{\li}^{q_0}\lf\{\int_{B\setminus 2B_j}
\lf[\int_{B_j}\frac1{[\lz(y,d(x,y))]^{1-\az}}\,d\mu(y)\r]^{q_0}\,d\mu(x)
+[\mu(4B_j)]^{q_0}\r\}\\
&&\hs\ls\sum_{j=1}^J\|a_j\|_{\li}^{q_0}
\lf\{\lf[\frac{\mu(B_j)}{(\lz(c_{B_j},r_{B_j}))^{1-\az}}\r]^{q_0}\mu(B)
+[\mu(4B_j)]^{q_0}\r\}\\
&&\hs\ls\sum_{j=1}^J\|a_j\|_{\li}^{q_0}
\lf\{[\mu(2B)]^{\az q_0}\mu(B)
+[\mu(4B_j)]^{q_0}\r\}
\ls \|a\|_{\li}^{q_0}[\mu(2B)]^{q_0}.
\end{eqnarray*}
Thus, \eqref{2.26} also holds true in this case.

Now we turn to prove (IV). By a standard argument (see \cite[Theorem 4.1]{hyy}
for the details),
it suffices to show that, for any $(\fz,1)_\lz$-atomic block $b$,
\begin{equation}\label{2.34}
\|T_\az b\|_{L^{q_0}(\mu)}\ls|b|_{H^{1,\fz}_{{\rm atb}}(\mu)}.
\end{equation}
Assume that $\supp b\st R$ and $b=\sum_{j=1}^2\lz_ja_j$, where, for $j\in\{1,2\}$,
$a_j$ is a function supported in $B_j\st R$ such that $\|a_j\|_{\li}
\le[\mu(4B_j)]^{-1}K^{-1}_{B_j,R}$ and $|\lz_1|+|\lz_2|\sim|b|_{H^{1,\fz}_{\rm atb}(\mu)}$.
Write
$$\int_\cx|T_\az b(x)|^{q_0}\,d\mu(x)=\int_{2R}|T_\az b(x)|^{q_0}\,d\mu(x)
+\int_{\cx\setminus 2R}\cdots=:{\rm L_1}+{\rm L_2}.$$
For ${\rm L_1}$, we see that
$${\rm L_1}\ls\sum_{j=1}^2|\lz_j|^{q_0}\int_{2B_j}|T_\az a_j(x)|^{q_0}\,d\mu(x)
+\sum_{j=1}^2|\lz_j|^{q_0}\int_{2R\setminus 2B_j}\cdots
=:{\rm L_{1,1}}+{\rm L_{1,2}}.$$
From \eqref{2.26}, $\|a_j\|_{\li}\ls[\mu(4B_j)]^{-1}K_{B_j,R}^{-1}$ for
$j\in\{1,2\}$, and Definition \ref{d1.11}(iii), it follows that
$${\rm L_{1,1}}\ls\sum_{j=1}^2|\lz_j|^{q_0}\|a_j\|_{\li}^{q_0}[\mu(4B_j)]^{q_0}
\ls\sum_{j=1}^2|\lz_j|^{q_0}\ls|b|^{q_0}_{H^{1,\fz}_{\rm atb}(\mu)}.$$
For ${\rm L_{1,2}}$, by \eqref{1.7}, Minkowski's inequality, \eqref{1.2}, \eqref{1.3},
(ii) and (iv) of Lemma \ref{l2.1}, the fact that
$\|a_j\|_{\li}\ls[\mu(4B_j)]^{-1}K_{B_j,R}^{-1}$ and Definition \ref{d1.11}(iii), we see that
\begin{eqnarray*}
{\rm L_{1,2}}&&\ls\sum_{j=1}^2|\lz_j|^{q_0}
\int_{2R\setminus 2B_j}\lf\{\int_{B_j}
\frac{|a_j(y)|}{[\lz(x,d(x,y))]^{1-\az}}\,d\mu(y)\r\}^{q_0}\,d\mu(x)\\
&&\ls\sum_{j=1}^2|\lz_j|^{q_0}
\lf\{\int_{B_j}|a_j(y)|\lf[\int_{2R\setminus 2B_j}
\frac1{\lz(x,d(x,y))}\,d\mu(x)\r]^{1/q_0}\,d\mu(y)\r\}^{q_0}\\
&&\ls\sum_{j=1}^2|\lz_j|^{q_0}
[\mu(B_j)]^{q_0}\|a_j\|_{\li}^{q_0}\int_{2R\setminus 2B_j}
\frac1{\lz(c_{B_j},d(x,c_{B_j}))}\,d\mu(x)\\
&&\ls\sum_{j=1}^2|\lz_j|^{q_0}
[\mu(B_j)]^{q_0}\|a_j\|_{\li}^{q_0}K_{B_j,R}
\ls\sum_{j=1}^2|\lz_j|^{q_0}\ls|b|^{q_0}_{H^{1,\fz}_{{\rm atb}}(\mu)}.
\end{eqnarray*}
Therefore, ${\rm L}_1\ls|b|^{q_0}_{H^{1,\fz}_{{\rm atb}}(\mu)}$.

On the other hand, from the fact that $\int_{\cx}b(y)\,d\mu(y)=0$, \eqref{1.8} and
Definition \ref{d1.11}(iii), we deduce that
\begin{eqnarray*}
{\rm L_2}&&\le\int_{\cx\setminus 2R}
\lf[\int_R|K_\az(x,y)-K_\az(x,c_R)||b(y)|\,d\mu(y)\r]^{q_0}\,d\mu(x)\\
&&\ls\lf[\int_R|b(y)|\,d\mu(y)\r]^{q_0}\sum_{i=1}^\fz\int_{2^{i+1}R\setminus 2^iR}
\frac{r_R^{\dz q_0}}{\lz(c_R,d(x,c_R))[d(x,c_R)]^{\dz q_0}}\,d\mu(x)\\
&&\ls(|\lz_1|+|\lz_2|)^{q_0}\sum_{i=1}^\fz 2^{-i\dz q_0}
\ls|b|^{q_0}_{H^{1,\fz}_{{\rm atb}}(\mu)},
\end{eqnarray*}
which, together with the estimate for ${\rm L}_1$, implies \eqref{2.34} and hence
completes the proof of (III)$\Rightarrow$(IV).

(IV)$\Rightarrow$(V) is obvious, the details being omitted.

(V)$\Rightarrow$(I) We first claim that, for any ball $B$ and $f\in\lon$ with
bounded support in $(6/5)B$,
\begin{equation}\label{2.35}
\frac1{\mu(6B)}\int_B|T_\az f(y)|\,d\mu(y)\ls\|f\|_{L^{1/\az}(\mu)}.
\end{equation}
Assume first that $r_B\le{\rm diam}(\supp\mu)/40$. We consider the same construction
in the proof of (III)$\Rightarrow$(IV). Let $B$, $B_0$ and $S$ be the same as there.
We know that $B,\,B_0\st2S$, $B_0$ is doubling, $K_{B,2S}\ls1$, $K_{B_0,2S}\ls1$
and ${\rm dist}(B_0,B)\gs r_B$.
Let $g=f+C_{B_0}\chi_{B_0}$, where $C_{B_0}$ is a constant such that
$\int_\cx g(x)\,d\mu(x)=0$. Then $g$ is an $(\fz,1)_\lz$-atomic block supported in $R$.
It is easy to show that
\begin{equation}\label{2.36}
\|g\|_{\hon}\ls[\mu(6B)]^{1/q_0}\|f\|_{L^{1/\az}(\mu)},
\end{equation}
where $q_0:=1/(1-\az)$.
For $y\in B$, by \eqref{1.7}, the fact that ${\rm dist}(B_0,B)\gs r_B$,
\eqref{1.3}, $\int_\cx g(x)\,d\mu(x)=0$, H\"older's inequality and \eqref{1.2},  we have
\begin{eqnarray}\label{2.37}
&&|T_\az(C_{B_0}\chi_{B_0})(y)|\\
&&\noz\hs\ls|C_{B_0}|\int_{B_0}\frac1{[\lz(y, d(x,y))]^{1-\az}}\,d\mu(x)
\ls\frac{|C_{B_0}|\mu(B_0)}{[\lz(c_B, r_B)]^{1-\az}}\\
&&\noz\hs\ls\|f\|_{\lon}\frac1{[\lz(c_B, r_B)]^{1-\az}}
\ls\lf[\frac{\mu((6/5)B)}{\lz(c_B, r_B)}\r]^{1-\az}\|f\|_{L^{1/\az}(\mu)}
\ls\|f\|_{L^{1/\az}(\mu)}.
\end{eqnarray}
Denote $\|g\|_{\hon}[\mu(B)]^{-1/q_0}$ simply by $E$. Then by (V) of Theorem \ref{t1.13}
and \eqref{2.36}, we conclude that
\begin{eqnarray}\label{2.38}
\int_B|T_\az g(y)|\,d\mu(y)
&&=\int_0^{E}\mu(\{y\in B:\ |T_\az g(y)|> t\})\,d t+\int_{E}^\fz\cdots\\
&&\ls E\mu(B)+\int_{E}^\fz\|g\|_{\hon}^{q_0} t^{-q_0}\,d t
\ls\mu(6B)\|f\|_{L^{1/\az}(\mu)}.\noz
\end{eqnarray}
The estimates \eqref{2.37} and \eqref{2.38} imply \eqref{2.35} in this case.

If $r_B>{\rm diam}(\supp\mu)/40$, by an argument similar to that used
in the proof of
\eqref{2.26} in the case of $r_B>{\rm diam}(\supp\mu)/40$, we can prove that
\eqref{2.35} also holds true in this case.

Now we turn to prove (I). By Theorem \ref{t2.7}, we only need to prove that
$T_\az$ is bounded from $L^{1/\az}(\mu)$ into $\rbmo$. Repeating the proofs
of \eqref{2.24} and \eqref{2.25} step by step, only needing to replace the
$(\lon,L^{1/(1-\az),\,\fz}(\mu))$-boundedness of $T_\az$ by \eqref{2.35}
when estimating H, we then know that $T_\az$ is bounded from $L^{1/\az}(\mu)$ into $\rbmo$,
which completes the proof that (V) implies (I) and hence
the proof of Theorem \ref{t1.13}.
\end{proof}

\section{Proofs of Theorems \ref{t1.15} and \ref{t1.19}\label{s3}}

\hskip\parindent In order to prove Theorem \ref{t1.15}, we need a technical lemma which
is a variant over non-homogeneous metric measure spaces of \cite[Lemma 2]{cs}.
\begin{lemma}\label{l3.1}
Let $\az\in (0,1)$, $p\in(1,1/\az)$, $\rho\in [5,\fz)$, $r\in(p,1/\az)$ and $1/q=1/r-\az$.
Then there exists a positive constant $C$ such that, for all $f\in L^r(\mu)$,
$$\|M_{p,\rho}^{(\az)}f\|_{L^{q}(\mu)}\le C\|f\|_{L^{r}(\mu)},$$
where
$$M_{p,\rho}^{(\az)}f(x):=\sup_{ Q\ni x}\lf\{\frac{1}{[\mu(\rho Q)]^{1-\az p}}
\int_{Q}|f(y)|^{p}\,d\mu(y)\r\}^{1/p}$$
and the supremum is taken over all balls $Q\ni x$.
\end{lemma}

\begin{proof}
We first prove that
\begin{equation}\label{3.1}
\mu\lf(\lf\{x\in\cx:\ M_{p,\rho}^{(\az)}f(x)> t\r\}\r)
\ls\lf[\|f\|_{L^{p}(\mu)}/ t\r]^{p/(1-\az p)}.
\end{equation}
Let $E:=\{x\in\cx:\ M_{p,\rho}^{(\az)}f(x)> t\}$.

For any $x\in E$, there exists a ball $Q_{x}$ containing $x$ such that
\begin{equation}\label{3.2}
\frac{1}{[\mu(\rho Q_{x})]^{1-\az p}}\int_{Q_{x}}|f(y)|^{p}\,d\mu(y)> t^{p}.
\end{equation}
By \cite[Theorem 1.2]{he} and \cite[Lemma 2.5]{h10}, there exist countable
disjoint subsets $\{Q_{j}\}_{j}$
of $\{Q_{x}:x\in E\}$ such that $E\st\cup_{j}\rho Q_{j}$. Let
$q:=p/(1-\az p)$. Then $p/q\le 1$. Hence, by \eqref{3.2} and $p/q=1-\az p$,
we see that
$$[\mu(E)]^{p/q}\le [\mu\lf(\cup_{j}\rho Q_{j}\r)]^{p/q}
\le \sum_{j}[\mu(\rho Q_{j})]^{p/q}\le \sum_{j}\frac{1}{ t^{p}}
\int_{Q_{j}}|f(y)|^{p}\,d\mu(y)
\le \frac{\|f\|_{L^{p}(\mu)}^{p}}{t^{p}}.$$
Hence $\mu(E)\ls  t^{-q}\|f\|_{L^{p}(\mu)}^{q}$, namely, \eqref{3.1} holds true.

Notice that, if $p<s<1/\az$, by using H\"older's inequality, we have
$M_{p,\rho}^{(\az)}f\le M_{s,\rho}^{(\az)}f.$
Hence, by the proceeding arguments, we see that
$\mu(E)\ls[\frac{1}{ t}\|f\|_{L^{s}(\mu)}]^{s/(1-\az s)},$
which, together with \eqref{3.1} and the Marcinkiewicz interpolation theorem,
further implies the desired result and hence completes the proof of Lemma \ref{l3.1}.
\end{proof}

\begin{remark}\rm\label{r3.2}
Let $\az\in (0,1)$. By Lemma \ref{l3.1}, the maximal operators $M_{r,\rho}^{(\az)}$
($r<p<1/\az$) and
$M_{(\rho)}^{(\az)}:=M_{1,\rho}^{(\az)}$ are bounded from $\lp$
to $\lq$ for $p\in (r,1/\az)$ and $1/q=1/p-\az$.
\end{remark}

Now we introduce the  fractional coefficient
${\wz K}_{B,S}^{(\az)}$ adapted from
\cite{cs}.

\begin{definition}\label{d3.3}
For any two balls $B\st S$, ${\wz K}_{B,S}^{(\az)}$ is
defined by
\begin{equation*}
{\wz K}_{B,S}^{(\az)}
:=1+\sum_{k=1}^{N_{B,S}}\lf[\frac{\mu(6^{k}B)}
{ \lz(x_{B},6^{k}r_{B})}\r]^{1-\az},
\end{equation*}
where $\az\in [0,1)$ and $N_{B,S}$ is defined as in Remark \ref{r1.6}.
\end{definition}
Now we give out some simple properties of ${\wz K}_{B,S}^{(\az)}$, which
are completely analogous to \cite[Lemma 3]{cs}. We omit the details;
see \cite[Lemma 3.5]{fyy} for the proofs of the case that $\az=0$.
\begin{lemma}\label{l3.4}
Let $\az\in[0,1)$.
\vspace{-0.25cm}
\begin{itemize}
 \item[\rm(i)] For all balls $B\st R\st S$, ${\wz K}_{B,R}^{(\az)}
 \le 2{\wz K}_{B,S}^{(\az)}$.
\vspace{-0.25cm}
 \item[\rm(ii)] For any $\rho\in [1,\fz)$, there exists a positive constant
$C_{(\rho)}$, depending only on $\rho$, such that, for all balls $B\st S$ with
$r_S\le \rho r_B$, $\wz K_{B,S}^{(\az)}\le C_{(\rho)}$.
\vspace{-0.25cm}
 \item[\rm(iii)] There exists a positive constant
$C_{(\az)}$, depending on $\az$, such that, for all balls $B$,
${\wz K}_{B,\wz B}^{(\az)}\le C_{(\az)}$.
\vspace{-0.25cm}
\item[\rm(iv)] There exists a positive constant $c$, depending on $C_\lz$ and $\az$,
such that, for all balls
$B\st R\st S$,
${\wz K}_{B,S}^{(\az)}\le {\wz K}_{B,R}^{(\az)}+c{\wz K}_{R,S}^{(\az)}.$
\vspace{-0.25cm}
 \item[\rm(v)] There exists a positive constant $\wz c$, depending on $C_\lz$ and $\az$,
 such that, for all balls
$B\st R\st S$,
${\wz K}_{R,S}^{(\az)}\le \wz c{\wz K}_{B,S}^{(\az)}.$
\end{itemize}
\end{lemma}

Now we introduce the sharp maximal operator ${\wz M}^{\#,\,\az}$ associated with
${\wz K}_{B,S}^{(\az)}$.

\begin{definition}\label{d3.5} Let $\az\in[0,1)$.
For all $f\in L^1_\loc(\mu)$ and $x\in\cx$, the \emph{sharp maximal function ${\wz M}^{\#,\,\az}f(x)$}
of $f$ is defined by
$${\wz M}^{\#,\,\az}f(x):=\sup_{B\ni x}\frac1{\mu(6B)}\int_B|f(x)-m_{\wz B}f|\,d\mu(x)
+\sup_{(Q,R)\in \bdz_x}\frac{|m_Qf-m_Rf|}{{\wz K}_{Q,R}^{(\az)}},$$
where $\bdz_{x}:= \{(Q,R):\ \,x\in Q\st R\ {\rm and}\
Q,\,R\,\ {\rm are\ doubling\ balls}\}$.
\end{definition}

Similar to \cite[Theorem 4.2]{bd}, we have the following lemma.

\begin{lemma}\label{l3.6}
Let $f\in L^1_\loc(\mu)$ satisfy that $\int_\cx f(x)\,d\mu(x)=0$ when
$\|\mu\|:=\mu(\cx)<\fz$. Assume that, for some $p\in (1,\fz)$, $\inf\{1,Nf\}\in \lp$.
Then there exists a positive constant $C$, independent of $f$, such that
$\|Nf\|_\lp\le C\|\wz M^{\#,\az}f\|_\lp.$
\end{lemma}

The following two lemmas are completely analogous to \cite[Lemmas 5 and 6]{cs},
the details being omitted.

\begin{lemma}\label{l3.7}
For any $\az\in [0,1)$, there exists some positive constant $P_\az$ (big enough),
depending only on $C_\lz$ in \eqref{1.2} and $\az$, such that, if $m\in\nn$,
$B_1\st\cdots\st B_m$ are concentric balls with
${\wz K}_{B_i,B_{i+1}}^{(\az)}>P_\az$
for $i\in\{1,\ldots,m-1\}$, then there exists a positive constant $C$,
depending only on $C_\lz$ and $\az$, such that
$\sum_{i=1}^{m-1}{\wz K}_{B_i,B_{i+1}}^{(\az)}\le C{\wz K}_{B_1,B_m}^{(\az)}.$
\end{lemma}

\begin{lemma}\label{l3.8}
For any $\az\in [0,1)$, there exists some positive constant
$\wz{P_{\az}}$ (large enough), depending on $C_{\lz}$, $\bz_6$ as in \eqref{1.2} with
$\eta=6$ and
$\az$, such that, if $x\in \cx$ is some fixed point and $\{f_{B}\}_{B\ni x}$
is a collection of numbers such that $|f_{B}-f_{S}|\le
{\wz K}_{B,S}^{(\az)}C_{x}$ for all doubling balls $B\st S$ with $x\in B$
satisfying ${\wz K}_{B,S}^{(\az)}\le \wz{P_{\az}}$, then there exists a positive
constant $C$, depending on $C_{\lz}$, $\bz_6$, $\az$ and $\wz{P_{\az}}$, such that
$|f_{B}-f_{S}|\le C_4{\wz K}_{B,S}^{(\az)}C_{x}$ for all doubling
balls $B\st S$ with $x\in B$, where $C_{x}$ is a positive constant, depending
on $x$, and $C_4$ a positive constant depending only on $C_\lz$, $\bz_6$ and $\az$.
\end{lemma}

The following theorem is adapted from \cite[Theorem 1]{cs}.

\begin{theorem}\label{t3.9}
Let $b\in \rbmo$ and $T_{\az}$ for $\az\in (0,1)$ be as in \eqref{1.9}
with kernel $K_\az$ satisfying \eqref{1.7} and \eqref{1.8},
which is bounded from $\lp$ into $\lq$ for all $p\in (1,1/\az)$
and $1/q=1/p-\az$. Then the commutator
$[b,T_{\az}]$ satisfies that there exists a positive constant $C$ such that,
for all $f\in\lp$,
$\|[b,T_{\az}]f\|_{L^{q}(\mu)}\le C\|b\|_{\rbmo}\|f\|_{L^{p}(\mu)}.$
\end{theorem}

\begin{proof}
The case that $\mu(\cx)<\fz$ can be proved by a way similar to the proof of
\cite[Theorem 3.10]{fyy}. Thus, without loss of generality, we may assume that
$\mu(\cx)=\fz$. Let $p\in (1,1/\az)$.
We first claim that, for all $r\in(1,\fz)$, $f\in\lp$ and $x\in \cx$,  \begin{equation}\label{3.3}
{\wz M}^{\#,\,\az}([b,T_{\az}]f)(x)\ls
\|b\|_{\rbmo}\lf\{M_{r,5}^{(\az)}f(x)+M_{r,6}(T_{\az}f)(x)
+T_{\az}(|f|)(x)\r\}.
\end{equation}
Once \eqref{3.3} is proved, taking $1<r<p<1/\az$,
by Lemma \ref{l2.5}(ii), Lemma \ref{l3.6}, an argument similar
to that used in the proof of \cite[Theorem 3.10]{fyy},
and Remark \ref{r3.2}, we conclude that
\begin{eqnarray*}
\|[b,T_{\az}]f\|_{L^{q}(\mu)}&&\le
\|N([b,T_{\az}]f)\|_{L^{q}(\mu)}\ls
\|{\wz M}^{\#,\,\az}([b,T_{\az}]f)\|_{L^{q}(\mu)}\\
&&\ls
\|b\|_{\rbmo}\lf\{\|M_{r,5}^{(\az)}f\|_{L^{q}(\mu)}
+\|M_{r,6}(T_{\az}f)\|_{L^{q}(\mu)}+\|T_{\az}f\|_{L^{q}(\mu)}\r\}\\
&&\ls \|b\|_{\rbmo}\|f\|_{L^{p}(\mu)},
\end{eqnarray*}
which is just the desired conclusion.

To show \eqref{3.3}, by Definition \ref{d1.9}, there exists
a family of numbers, $\{b_{Q}\}_{Q}$, such that, for any ball $Q$,
$$\int_{Q}|b(y)-b_{Q}|\,d\mu(y)\ls \mu(6Q)\|b\|_{\rbmo}$$
and, for all balls $Q$, $R$ with $Q\st R$,
$|b_{Q}-b_{R}|\ls K_{Q,R}\|b\|_{\rbmo}.$
For any ball $Q$, let
$$h_{Q}:=m_{Q}(T_{\az}([b-b_{Q}]f\chi_{\cx\setminus (6/5)Q})).$$
Next we show that, for all $x$ and $Q$ with $Q\ni x$,
\begin{equation}\label{3.4}
\frac{1}{\mu(6Q)}\int_{Q}|[b,T_{\az}]f(y)-h_{Q}|\,d\mu(y)\ls
\|b\|_{\rbmo}\lf\{M_{p,5}^{(\az)}f(x)+M_{p,6}(T_{\az}f)(x)\r\}
\end{equation}
and, for all balls $Q$, $R$ with $Q\st R$ and $Q\ni x$,
\begin{equation}\label{3.5}
|h_{Q}-h_{R}|\ls
\|b\|_{\rm{RBMO}}(\mu)\lf\{M_{p,5}^{(\az)}f(x)+T_\az (|f|)(x)\r\}
K_{Q,R}{\wz K}_{Q,R}^{(\az)}.
\end{equation}

To prove \eqref{3.4}, for a fixed ball $Q$ and $x$ with $x\in Q$, we write
$[b,T_{\az}]f$ as
\begin{equation}\label{3.6}
[b,T_{\az}]f=[b-b_{Q}]T_{\az}f-T_{\az}([b-b_{Q}]f_{1})
-T_{\az}([b-b_{Q}]f_{2}),
\end{equation}
where $f_{1}:=f\chi_{(6/5)Q}$ and $f_{2}:=f-f_{1}$.

Let us first estimate the term $[b-b_{Q}]T_{\az}f$. By H\"older's
inequality and \cite[Corollary 6.3]{h10}, we see that
\begin{eqnarray}\label{3.7}
&&\frac{1}{\mu(6Q)}\int_{Q}|[b(y)-b_{Q}]T_{\az}f(y)|\,d\mu(y)\\
\noz &&\hs\le
\lf[\frac{1}{\mu(6Q)}\int_{Q}|b(y)-b_{Q}|^{p'}\,d\mu(y)\r]^{1/p'}
\lf[\frac{1}{\mu(6Q)}\int_{Q}|T_{\az}f(y)|^{p}\,d\mu(y)\r]^{1/p}\\
\noz &&\hs\ls \|b\|_{\rbmo}M_{p,6}(T_{\az}f)(x),
\end{eqnarray}
which is desired.

To estimate $T_\az([b-b_Q]f_1)$, take $s:=\sqrt{p}$ and $1/r:=1/s-\az$. From H\"older's
inequality, the $(L^s(\mu),L^r(\mu))$-boundedness of $T_\az$ and \cite[Corollary 6.3]{h10},
it follows that
\begin{eqnarray}\label{3.8}
&&\frac{1}{\mu(6Q)}\int_{Q}|T_{\az}([b-b_{Q}]f_{1})(y)|\,d\mu(y)\\
&&\noz\hs\le
\frac{[\mu(Q)]^{1-1/r}}{\mu(6Q)}\|T_{\az}([b-b_{Q}]f_{1})\|_{L^{r}(\mu)}
\ls \frac{[\mu(Q)]^{1-1/r}}{\mu(6Q)}\|(b-b_{Q})f_{1}\|_{L^{s}(\mu)}\\
&&\noz\hs\ls
\frac{1}{[\mu(6Q)]^{1/r}}\lf\{\int_{(6/5)Q}|b(y)-b_{Q}|^{ss'}
\,d\mu(y)\r\}^{\frac{1}{ss'}}\lf[\int_{(6/5)Q}|f(y)|^{p}\,d\mu(y)\r]^{\frac{1}{p}}\\
&&\noz\hs\ls \|b\|_{\rbmo}M_{p,5}^{(\az)}f(x),
\end{eqnarray}
which is desired.

By \eqref{3.6}, \eqref{3.7} and \eqref{3.8}, to obtain \eqref{3.4}, we still need to estimate the difference
$|T_{\az}([b-b_{Q}]f_{2})-h_{Q}|$ by writing that, for all $y_1,\,y_2\in Q$,
\begin{eqnarray*}
&&\lf|T_{\az}([b-b_{Q}]f_{2})(y_{1})-T_{\az}([b-b_{Q}]f_{2})(y_{2})\r|\\
&&\hs\ls\int_{6Q\setminus
(6/5)Q}|K_\az(y_1,z)-K_\az(y_2,z)|
|b(z)-b_{Q}||f(z)|\,d\mu(z)
\,d\mu(z)+\int_{\cx\setminus 6Q}\cdots\\
&&\hs=:{\rm I}_1+{\rm I}_2.
\end{eqnarray*}

Let $c_Q$ and $r_Q$ be the center and the radius of $Q$, respectively.
To estimate ${\rm I}_1$, from \eqref{1.7} and H\"older's inequality, together with
\eqref{1.2} and \eqref{1.3}, it follows that
\begin{eqnarray*}
{\rm I}_1&&\ls\int_{6Q\setminus (6/5)Q}\lf(\frac1{[\lz(y_1,d(y_1,z))]^{1-\az}}
+\frac1{[\lz(y_2,d(y_2,z))]^{1-\az}}\r)|f(z)||b(z)-b_Q|\,d\mu(z)\\
&&\ls\lf[\frac1{\mu(30Q)}\int_{6Q}|b(z)-b_Q|^{p'}\,d\mu(z)\r]^{1/p'}
\lf\{\frac1{[\mu(30Q)]^{1-\az p}}\int_{6Q}|f(z)|^p\,d\mu(z)\r\}^{1/p}\\
&&\ls\|b\|_{\rbmo}M^\az_{p,5}f(x),
\end{eqnarray*}
which is desired.

For any $y_{1},y_{2}\in Q$, by \eqref{1.8},
\eqref{1.3}, \eqref{1.2}, H\"older's inequality and \cite[Corollary 6.3]{h10},
we know that
\begin{eqnarray*}
{\rm I}_2&&\ls \int_{\cx\setminus
6Q}\frac{[d(y_{1},y_{2})]^{\dz}}{[d(y_{1},z)]^{\dz}
[\lz(y_{1},d(y_{1},z))]^{1-\az}}
|b(z)-b_{Q}||f(z)|\,d\mu(z)\\
&&\hs\ls
\sum_{k=1}^{\fz}\int_{2^{k}(6Q)\setminus
2^{k-1}(6Q)}
\frac{(2r_{Q})^{\dz}}{[2^{k-1}\times6r_{Q}]^{\dz}}
\frac{1}{[\lz(c_{Q},2^{k-1}\times6r_{Q})]^{1-\az}}|b(z)-b_{Q}||f(z)|\,d\mu(z)\\
\noz &&\hs\ls
\sum_{k=1}^{\fz}2^{-k\dz}\frac{1}{[\mu(2^{k}\times
30Q)]^{1-\az}}\lf[\int_{2^{k}(6Q)}|b(z)-b_{2^{k}(6Q)}||f(z)|\,d\mu(z)\r.\\
&&\lf.\hs\hs+k\|b\|_{\rbmo}
\int_{2^{k}(6Q)}|f(z)|\,d\mu(z)\r]\\
&&\hs\ls
\sum_{k=1}^{\fz}2^{-k\dz}\lf(\lf[\frac{1}{\mu(2^{k}\times
30Q)}\int_{2^{k}(6Q)}|b(z)-b_{2^{k}(6Q)}|^{p'}
\,d\mu(z)\r]^{\frac{1}{p'}}\r.\\
&&\lf.\hs\hs\times\lf\{\frac{1}{[\mu(2^{k}\times 30Q)]^{1-\az
p}}\int_{2^{k}(6Q)}|f(z)|^{p}\,d\mu(z)\r\}^{1/p}\r.\\
&&\lf.\hs\hs +k\|b\|_{\rbmo}\lf\{\frac{1}{[\mu(2^{k}\times 30Q)]^{1-\az
p}}\int_{2^{k}(6Q)}|f(z)|^{p}\,d\mu(z)\r\}^{1/p}\r)\\
&&\hs\ls
\sum_{k=1}^{\fz}(k+1)2^{-k\dz}\|b\|_{\rbmo}M_{p,5}^{(\az)}f(x)
\ls \|b\|_{\rbmo}M_{p,5}^{(\az)}f(x),
\end{eqnarray*}
where we used the fact that
$$|b_{Q}-b_{2^{k}(6/5)Q}|\ls K_{Q,2^{k}(6Q)}\|b\|_{\rbmo}
\ls k\|b\|_{\rbmo}.$$

Combining the estimates for ${\rm I}_1$ and ${\rm I}_2$, we see that, for all $y\in Q$,
$$|T_{\az}([b-b_{Q}]f_{2})(y)-h_{Q}|\ls \|b\|_{\rbmo}M_{p,5}^{(\az)}f(x).$$
Thus,
$$\frac{1}{\mu(6Q)}\int_{Q}|T_{\az}([b-b_{Q}]f_{2})(y)-h_{Q}|\,d\mu(y)
\ls \|b\|_{\rbmo}M_{p,5}^{(\az)}f(x),$$
which, together with \eqref{3.6}, \eqref{3.7} and \eqref{3.8}, implies \eqref{3.4}.

Now we show the regularity condition \eqref{3.5} for the
numbers $\{h_{Q}\}_{Q}$. Consider two balls $Q\st R$ with $x\in
Q$ and let $N:=N_{Q,R}+1$. Write $|h_{Q}-h_{R}|$ as
\begin{eqnarray*}
&&|m_{Q}(T_{\az}([b-b_{Q}]f\chi_{\cx\setminus
(6/5)Q}))-m_{R}(T_\az([b-b_{Q}]f\chi_{\cx\setminus (6/5)R}))|\\
&&\hs\le |m_{Q}(T_{\az}([b-b_{Q}]f\chi_{6Q\setminus (6/5)Q}))|
+|m_{Q}(T_{\az}([b_{Q}-b_{R}]f\chi_{\cx\setminus 6Q}))|\\
&&\hs\hs+|m_{Q}(T_{\az}([b-b_{R}]f\chi_{6^{N}Q\setminus 6Q}))|
+|m_{Q}(T_{\az}([b-b_{R}]f\chi_{\cx\setminus 6^{N}Q}))\\
&&\hs\hs-m_{R}(T_{\az}([b-b_{R}]f\chi_{\cx\setminus 6^{N}Q}))|
+|m_{R}(T_{\az}([b-b_{R}]f\chi_{6^{N}Q\setminus (6/5)R}))|\\
&&\hs =:\mathrm{U_{1}}+\mathrm{U_{2}}+\mathrm{U_{3}}+\mathrm{U_{4}}
+\mathrm{U_{5}}.
\end{eqnarray*}

Following the proof of \cite[Theorem 1]{cs}, it is easy to see that
$$\mathrm{U_{1}}+\mathrm{U_{4}}+\mathrm{U_{5}}\ls \|b\|_{\rbmo}
M_{p,5}^{(\az)}f(x)$$
and
$\mathrm{U_{2}}\ls K_{Q,R}\|b\|_{\rbmo}[T_{\az}(|f|)(x)+M_{p,5}^{(\az)}f(x)].$

Now we turn to the estimate for $\mathrm{U_{3}}$. For $y\in Q$,
by \eqref{1.7} and H\"older's inequality, we conclude that
\begin{eqnarray*}
&&|T_{\az}([b-b_{R}]f\chi_{6^{N}Q\setminus 6Q})(y)|\\
&&\hs\ls \sum_{k=1}^{N-1}\frac{1}{[\lz(x_{Q},6^{k}r_{Q})]^{1-\az}}
\int_{6^{k+1}Q\setminus 6^{k}Q}|b(y)-b_{R}||f(y)|\,d\mu(y)\\
&&\hs\ls
\sum_{k=1}^{N-1}\frac{1}{[\lz(x_{Q},6^{k}r_{Q})]^{1-\az}}
\lf[\int_{6^{k+1}Q}|b(y)-b_{R}|^{p'}\,d\mu(y)\r]^{1/p'}
\lf[\int_{6^{k+1}Q}|f(y)|^{p}\,d\mu(y)\r]^{1/p}.
\end{eqnarray*}
Notice that, by Minkowski's inequality and Lemma \ref{l2.1}(i), we see that
\begin{eqnarray*}
&&\lf[\int_{6^{k+1}Q}|b(y)-b_{R}|^{p'}\,d\mu(y)\r]^{1/p'}\\
&&\hs\le
\lf[\int_{6^{k+1}Q}|b(y)-b_{6^{k+1}Q}|^{p'}\,d\mu(y)\r]^{1/p'}
+\lf[\mu(6^{k+1}Q)\r]^{1/p'}|b_{6^{k+1}Q}-b_{R}|\\
&&\hs\ls
K_{Q,R}\|b\|_{\rbmo}\lf[\mu(5\times6^{k+1}Q)\r]^{1/p'}.
\end{eqnarray*}
Thus, by \eqref{1.7}, \eqref{1.3} and \eqref{1.2}, we conclude that
\begin{eqnarray*}
&&|T_{\az}([b-b_{R}]f\chi_{6^{N}Q\setminus 6Q})(y)|\\
&&\hs\ls K_{Q,R}\|b\|_{\rbmo}\sum_{k=1}^{N-1}
\frac{[\mu(5\times
6^{k+1}Q)]^{1-1/p}}{[\lz(x_{Q},6^{k}r_{Q})]^{1-\az}}
\lf[\int_{6^{k+1}Q}|f(y)|^{p}\,d\mu(y)\r]^{1/p}\\
&&\hs\ls
K_{Q,R}\|b\|_{\rbmo}\sum_{k=1}^{N_{Q,R}}
\lf[\frac{\mu(6^{k+2}Q)}{\lz(x_{Q},6^{k}r_{Q})}\r]^{1-\az}\\
&&\hs\hs\times\lf\{\frac{1}{[\mu(5\times 6^{k+1}Q)]^{1-\az
p}}\int_{6^{k+1}Q}|f(y)|^{p}\,d\mu(y)\r\}^{1/p}
\ls K_{Q,R}{\wz K}_{Q,R}^{(\az)}\|b\|_{\rbmo}M_{p,5}^{(\az)}f(x).
\end{eqnarray*}
Taking the mean over $Q$, we obtain
$\mathrm{U_{3}}\ls K_{Q,R}{\wz K}_{Q,R}^{(\az)}\|b\|_{\rbmo}M_{p,5}^{(\az)}f(x),$
which, together with the estimates $\mathrm{U_{1}}$, $\mathrm{U_{2}}$,
$\mathrm{U_{4}}$ and $\mathrm{U_{5}}$, further implies \eqref{3.5}.

By \eqref{3.4}, if $Q$ is a doubling ball and
$x\in Q$, we have
\begin{equation}\label{3.9}
|m_{Q}([b,T_{\az}]f)-h_{Q}|\ls
\|b\|_{\rbmo}\lf[M_{p,5}^{(\az)}f(x)+M_{p,6}(T_{\az}f)(x)\r].
\end{equation}
Since, for any ball $Q$ with $x\in Q$, $K_{Q,{\wz Q}}\le C$ and
${\wz K}_{Q,{\wz Q}}^{(\az)}\le C$, by \eqref{3.4},
\eqref{3.5} and \eqref{3.9}, we see that
\begin{eqnarray}\label{3.10}
&&\frac{1}{\mu(6Q)}\int_{Q}|[b,T_{\az}]f(y)-m_{\wz Q}([b,T_{\az}]f)|\,d\mu(y)\\
\noz &&\hs\le
\frac{1}{\mu(6Q)}\int_{Q}|[b,T_{\az}]f(y)-h_{Q}|\,d\mu(y)+|h_{Q}-h_{\wz{Q}}|
+|h_{\wz{Q}}-m_{\wz{Q}}([b,T_{\az}]f)|\\
\noz &&\hs\ls
\|b\|_{\rbmo}\lf\{M_{p,5}^{(\az)}f(x)+M_{p,6}(T_{\az}f)(x)
+T_{\az}(|f|)(x)\r\}.
\end{eqnarray}
On the other hand, for all doubling balls $Q\st R$ with $x\in Q$
such that ${\wz K}_{Q,R}^{(\az)}\le \wz{P_{\az}}$, where
$\wz{P_{\az}}$ is the constant as in Lemma \ref{l3.8},
by \eqref{3.5}, we have
$$|h_{Q}-h_{R}|\ls K_{Q,R}\|b\|_{\rbmo}\lf[M_{p,5}^{(\az)}f(x)
+T_{\az}(|f|)(x)\r]\wz{P_{\az}}.$$
Hence, by Lemma \ref{l3.8}, we know that, for all doubling balls $Q\st R$ with
$x\in Q$,
$$|h_{Q}-h_{R}|\ls {\wz K}_{Q,R}^{(\az)}\|b\|_{\rbmo}\lf[M_{p,5}^{(\az)}f(x)
+T_{\az}(|f|)(x)\r]$$
and, using \eqref{3.9}, we further obtain
\begin{eqnarray*}
&&|m_{Q}([b,T_{\az}]f)-m_{R}([b,T_{\az}]f)|\\
&&\hs\ls
{\wz K}_{Q,R}^{(\az)}\|b\|_{\rbmo}\lf\{M_{p,5}^{(\az)}f(x)
+M_{p,6}(T_{\az}f)(x)
+T_{\az}(|f|)(x)\r\},
\end{eqnarray*}
which, together with \eqref{3.10}, induces \eqref{3.3} and hence
completes the proof of Theorem \ref{t3.9}.
\end{proof}

To prove Theorem \ref{t1.15}, we need to recall some notation from \cite{hmy1}.
Let $C^k_i$ be as in Section \ref{s1}.
For any sequence $\vec{b}:=(b_{1},\ldots,b_{k})$ of functions and
all $i$-tuples $\sz:=\{\sz(1),\ldots,\sz(i)\}\in C_{i}^{k}$, let
$\vec b_\sz:=(b_{\sz(1)},\ldots, b_{\sz(i)})$ and
$$\|\vec{b}_{\sz}\|_{\rbmo}:=\prod_{j=1}^i\|b_{\sz(i)}\|_{\rbmo}.$$
For any $\sz\in
C_{i}^{k}$ and $z\in\cx$, let
$$\lf[m_{\wz B}(\vec{b})-\vec{b}(z)\r]_{\sz}:=\prod_{j=1}^i\lf[m_{\wz B}
(b_{\sz(j)})-b_{\sz(j)}(z)\r]$$
and
$T_{\az,\vec{b}_{\sz}}:=[b_{\sz(i)},[b_{\sz(i-1)},
\cdots,[b_{\sz(1)},T_\az]\cdots]].$
In particular, when $\sz:=\{1,\ldots,k\}$,
$T_{\az,\vec{b}_{\sz}}$ coincides with $T_{\az,\vec{b}}$ as in \eqref{1.11}.

Now we are ready to prove Theorem \ref{t1.15}.

\begin{proof}[Proof of Theorem \ref{t1.15}]
By Lemma \ref{l2.4}, it suffices
to prove that $T_{\az,\vec b}$
is bounded from $\lp$ into $\lq$ for all $p\in (1,1/\az)$ and $1/q=1/p-\az$.
We show this by induction on $k$.

By Theorem \ref{t3.9}, the conclusion is valid for
$k=1$.
Now assume that $k\ge 2$ is an integer and, for any $i\in
\{1,\ldots,k-1\}$ and any subset $\sz
=\{\sz(1),\ldots,\sz(i)\}$ of $\{1,\ldots,k-1\}$,
$T_{\az,\vec{b}_{\sz}}$ is bounded from $\lp$ to
$\lq$ for the same $p,\,q$ as those such that $T_\az$
is bounded from $\lp$ to $\lq$.

The case that $\mu(\cx)<\fz$ can be proved by a way similar to that
used in the proof of \cite[Theorem 3.10]{fyy}, the details being omitted.
Thus, without loss of generality, we may assume that
$\mu(\cx)=\fz$.
We first claim that, for any $r\in(1,\fz)$, $f\in\lp$ and $x\in\cx$,
\begin{eqnarray}\label{3.11}
{\wz M}^{\#,\,\az}(T_{\az,\vec{b}}f)(x)
&&\ls \|\vec{b}\|_{\rbmo}\lf[M_{r,6}T_{\az}f(x)+M_{r,5}^{(\az)}f(x)\r]\\
&&\noz\hs+\sum_{i=1}^{k-1}\sum_{\sz\in
C_{i}^{k}}\|\vec{b}_{\sz}\|_{\rbmo}M_{r,6}
(T_{\az,\vec{b}_{\sz'}}f)(x).
\end{eqnarray}

Once \eqref{3.11} is proved, by Lemmas \ref{l2.5} and \ref{l2.6},
an argument similar to that used in the proof of
Theorem \ref{t3.9}, and Remark \ref{r3.2}, we conclude that, for all $p\in(1,1/\az)$, $1/q=1/p-\az$ and $f\in\lp$,
\begin{eqnarray*}
\|T_{\az,\vec{b}}f\|_{\lq}&&\le \|N(T_{\az,\vec{b}}f)\|_{\lq}
\ls\lf\|\widetilde{M}^{\#}(T_{\az,\vec{b}}f)\r\|_{\lq}\\
&&\ls \|\vec{b}\|_{\rbmo}\lf[\|M_{r,6}(Tf)\|_{\lq}
+\|M_{r,5}(f)\|_{\lq}\r]\\
&&\hs+\sum_{i=1}^{k-1}\sum_{\sz\in
C_{i}^{k}}\|\vec{b}_{\sz}\|_{\rbmo}
\|M_{r,6}(T_{\vec{b}_{\sz'}}f)\|_{\lq}\\
&&\ls \|\vec{b}\|_{\rbmo}
\lf[\|Tf\|_{\lq}+\|f\|_{\lp}
+\sum_{i=1}^{k-1}\sum_{\sz\in
C_{i}^{k}}\|T_{\vec{b}_{\sz'}}f\|_{\lq}\r]\\
&&\ls \|\vec{b}\|_{\rbmo}\|f\|_{\lp},
\end{eqnarray*}
which is desired.

As in the proof of \cite[Theorem 2]{hmy1}, to prove \eqref{3.11}, it suffices to
show that, for all $x$ and $B$ with $B\ni x$,
\begin{eqnarray}\label{3.12}
\frac{1}{\mu(6B)}\int_{B}|T_{\az,\vec{b}}f(y)-h_{B}|\,d\mu(y)
&&\ls\|\vec{b}\|_{\rbmo}\lf[M_{r,6}(T_{\az}f)(x)+M_{r,5}^{(\az)}f(x)\r]\\
\noz&&\hs+\sum_{i=1}^{k-1}\sum_{\sz\in
 C_{i}^{k}}\|\vec{b}_\sz\|_{\rbmo}M_{r,6}(T_{\az,\vec{b}_{\sz'}}f)(x)
\end{eqnarray}
and, for an arbitrary ball $Q$, a doubling ball $R$ with $Q\st R$ and $x\in Q$, \begin{eqnarray}\label{3.13}
|h_{Q}-h_{R}| &&\ls\lf[{\wz K}_{Q,R}\r]^{k}{\wz K}^{(\az)}_{Q,R}
\Bigg\{\|\vec{b}\|_{\rbmo}\{M_{r,6}T_{\az}f(x)
+M_{r,5}^{(\az)}f(x)\}\\
\noz&&\lf.\hs+\sum_{i=1}^{k-1}\sum_{\sz\in
C_{i}^{k}}\|\vec{b}_{\sz}\|_{\rbmo}
M_{r,6}(T_{\az,\vec{b}_{\sz'}}f)(x)\r\},
\end{eqnarray}
where
$$h_{Q}:=m_{Q}\lf(T_{\az}\lf(\prod_{i=1}^k[m_{\wz{Q}}(b_{i})-b_{i}]
f\chi_{\cx\setminus \frac{6}{5}Q}\r)\r)$$ and
$$h_{R}:=m_{R}\lf(T_{\az}\lf(\prod_{i=1}^k[m_{R}(b_{i})-b_{i}]
f\chi_{\cx\setminus \frac{6}{5}R}\r)\r).$$

Let us first prove \eqref{3.12}. With the aid of the
formula that, for all $y,\,z\in\cx$,
\begin{eqnarray*}
\prod_{i=1}^{k}[m_{\wz{Q}}(b_{i})-b_{i}(z)]=\sum_{i=0}^{k}\sum
\limits_{\sz\in C_{i}^{k}}[b(y)-b(z)]_{\sz'}
[m_{\wz{Q}}(b)-b(y)]_{\sz},
\end{eqnarray*}
where, if $i=0$, then
$\sz'=\{1,\ldots,k\}$ and $\sz=\emptyset$,
$[m_{\wz{Q}}(b)-b(y)]_{\emptyset }=1$, it is easy to see that, for all $y\in\cx$,
$$T_{\az,\vec{b}}f(y)=T_{\az}\lf(\prod _{i=1}^{k}[m_{\wz{Q}}(b_{i})-b_{i}]f\r)(y)
-\sum_{i=1}^{k}\sum_{\sz\in
C_{i}^{k}}[m_{\wz{Q}}(b)-b(y)]_{\sz}
T_{\az,\vec{b}_{\sz\prime}}f(y),$$
where, if $i=k$, $T_{\az,\vec{b}_{\sz'}}f(y):=T_{\az}(|f|)(y)$. Therefore,
for all balls $Q\ni x$, we have
\begin{eqnarray*}
&&\frac{1}{\mu(6Q)}\int_{Q}\lf|T_{\az,\vec{b}}f(y)-h_{Q}\r|\,d\mu(y)\\
&&\hs\le\frac{1}{\mu(6Q)}\int_{Q}\lf|T_{\az}\lf(\prod_{i=1}^{k}
[m_{\wz{Q}}(b_{i})-b_{i}]
f\chi_{\frac{6}{5}{Q}}\r)(y)\r|\,d\mu(y)\\
&&\hs\hs+\sum_{i=1}^{k}\sum_{\sz \in C_{i}^{k}}
\frac{1}{\mu(6Q)}\int_{Q}\lf|\lf[m_{\wz{Q}}(b)-b(y)\r]_{\sz
}\r|\lf|T_{\az,\vec{b}_{\sz'}}f(y)\r|\,d\mu(y)\\
&&\hs\hs+\frac{1}{\mu(6Q)}\int_{Q}\lf|T_{\az}\lf(\prod_{i=1}^{k}
\lf[m_{\wz{Q}}(b_{i})-b_{i}\r]
f\chi_{\cx\setminus\frac{6}{5}Q}\r)(y)-h_{Q}\r|\,d\mu(y)
=:\mathrm{I}_1+\mathrm{I}_2+\mathrm{I}_3.
\end{eqnarray*}

Take $1/s^2=1/r-\az$. Using the boundedness of $T_{\az}$ from $L^{s/(1+s\az)}(\mu)$
into $L^s(\mu)$ for $s\in (1,\fz)$ and some arguments
similar to those used in the proofs of
\cite[Theorem 1.1]{hmy1} and \cite[Theorem 1.9]{fyy}, we conclude that, for all $x\in\cx$,
$\mathrm{I}_{1}\ls \|\vec{b}\|_{\rbmo}M_{r,5}^{(\az)}f(x),$
$$\mathrm{I}_{2}\ls
\sum_{i=1}^{k}\sum_{\sz\in
C_{i}^{k}}\|\vec{b}_{\sz}\|_{\rbmo}M_{r,6}
\lf(T_{\az,\vec{b}_{\sz'}}f\r)(x)$$
and $\mathrm{I}_{3}\ls \|\vec{b}_{\sz}\|_{\rbmo}M_{r,5}^{(\az)}f(x),$
which imply \eqref{3.12}.

Now we turn to prove \eqref{3.13}. Let $Q$ be an arbitrary ball and $R$
a doubling ball in $\cx$ such that $x\in Q\st R$.
Denote $N_{Q,R}+1$ simply by $N$. Write
\begin{eqnarray*}
&&|h_Q-h_R|\\
&&\hs\le
\lf|m_{R}\lf[T_{\az}\lf(\prod_{i=1}^{k}\lf[m_{\wz{Q}}(b_{i})-b_{i}\r]
f\chi_{\cx\setminus 6^{N}Q}\r)\r]
-m_{Q}\lf[T_{\az}\lf(\prod_{i=1}^{k}\lf[m_{\wz{Q}}(b_{i})-b_{i}\r]f
\chi_{\cx\setminus 6^{N}Q}\r)\r]\r|\\
&&\hs\hs+\lf|m_{R}\lf[T_{\az}\lf(\prod_{i=1}^{k}\lf[m_{\wz{Q}}(b_{i})-b_{i}\r]
f\chi_{\cx\setminus 6^{N}Q}\r)\r]-m_{R}\lf[T_{\az}\lf(\prod_{i=1}^{k}\lf[m_{R}(b_{i})-b_{i}\r]
f\chi_{\cx\setminus 6^{N}Q}\r)\r]\r|\\
&&\hs\hs+\lf|m_{Q}\lf[T_{\az}\lf(\prod_{i=1}^{k}\lf[m_{\wz{Q}}(b_{i})-b_{i}\r]f
\chi_{6^{N}Q\setminus\frac{6}{5}Q}\r)\r]\r|\\
&&\hs\hs+\lf|m_{R}\lf[T_{\az}\lf(\prod_{i=1}^{k}\lf[m_{R}(b_{i})-b_{i}\r]f
\chi_{6^{N}Q\setminus\frac{6}{5}R}\r)\r]\r|
=:\mathrm{L}_{1}+\mathrm{L}_{2}+\mathrm{L}_{3}+\mathrm{L}_{4}.
\end{eqnarray*}

An estimate similar to that for $\mathrm{I}_{3}$, together with
$K_{Q,R}\ls\wz{K}_{Q,R}$, we see that, for all $x\in\cx$,
$\mathrm{L}_{1}\ls[{\wz K}_{Q,R}]^{k}
\|\vec{b}\|_{RBMO(\mu)}M_{r,5}^{(\az)}f(x).$

By some arguments similar to those used
in the proofs of \cite[Theorem 1.1]{hmy1} and \cite[Theorem 1.9]{fyy},
we easily see that, for all $x\in\cx$,
\begin{eqnarray*}
\mathrm{L}_{2}&&\ls
\lf[{\wz K}_{Q,R}\r]^{k}\lf\{\sum_{i=1}^{k-1}\sum_{\sz\in C_{i}^{k}}
\|\vec{b}_{\sz'}\|_{\rbmo}M_{r,6}\lf(T_{\az,\vec{b}}f(x)\r)\r.\\
&&\hs+\|\vec{b}\|_{\rbmo}M_{r,6}
(T_{\az}f)(x)+\|\vec{b}\|_{\rbmo}M_{r,5}^{(\az)}f(x)\Bigg\},
\end{eqnarray*}
$\mathrm{L}_{3}\ls[{\wz K}_{Q,R}]^k
{\wz K}_{Q,R}^{(\az)}\|\vec{b}\|_{\rbmo}M_{r,5}^{(\az)}f(x)$
and
$\mathrm{L}_{4}\ls \|\vec{b}\|_{\rbmo}M_{r,5}^{(\az)}f(x).$

Combining the estimates for $\mathrm{L}_{1}$, $\mathrm{L}_{2}$,
$\mathrm{L}_{3}$ and $\mathrm{L}_{4}$, we then obtain \eqref{3.13} and hence
complete the proof of Theorem \ref{t1.13}.
\end{proof}

Now we are ready to prove Theorem \ref{t1.19}. In what follows, for any $k\in\nn$
and $i\in\{1,\ldots,k\}$,
let $C^k_i$ be as in the introduction.
For all sequences of numbers, $r:=(r_{1},\ldots,r_{k})$, and $i$-tuples
$\sz:=\{\sz(1),\ldots,\sz(i)\}\in C_{i}^{k}$, let $\vec{b}$ and
$\vec{b}_\sz$ be as in Theorem \ref{t1.15},
$$\|\vec{b}_{\sz}\|_{{\rm Osc}_{\exp L^{r_\sz}}(\mu)}
:=\prod_{j=1}^i\|b_{\sz(j)}\|_{{\rm Osc}_{\exp L^{r_{\sz(j)}}}(\mu)}$$
and, in particular,
$$\|\vec b\|_{{\rm Osc}_{\exp L^r}(\mu)}
:=\prod_{j=1}^k\|b_j\|_{\osj}.$$

Then we prove Theorem \ref{t1.19}.
\begin{proof}[Proof of Theorem \ref{t1.19}]
Without loss of generality, by homogeneity, we may assume that $\|f\|_{\lon}=1$ and
$\|b_i\|_{\osi}=1$ for all $i\in\{1,\ldots,k\}$.

We prove the theorem by two steps: $k=1$ and $k>1$.

Step i) $k=1$. It is easy to see that the conclusion of Theorem \ref{t1.19} automatically
holds true if
$ t\le\bz_6\|f\|_{\lon}/\mu(\cx)$ when $\mu(\cx)<\fz$. Thus, we only need to
 deal with the case that $ t>\bz_6\|f\|_{\lon}/\mu(\cx)$. For any given
bounded function $f$ with bounded support, $q_0:=1/(1-\az)$
and any $ t>\bz_6\|f\|_{\lon}/\mu(\cx)$, applying Lemma \ref{l2.6}
to $f$ with $ t$ replaced by $ t^{q_0}$, and letting $R_j$ be as in Lemma \ref{l2.6}(iii),
we see that $f=g+h$, where
$g:=f\chi_{\cx\setminus \cup_j6B_j}+\sum_j\vz_j$ and $h:=\sum_j(\oz_jf-\vz_j)=:\sum_j h_j$.
Let $p_1\in(1,1/\az)$ and $1/q_1:=1/p_1-\az$.
By \eqref{2.7}, we easily know that $\|g\|_\li\ls  t^{q_0}$. From this,
the boundedness of $T_{\az}$ from $\lpo$ to $\lqo$ and
\eqref{2.19}, it follows that
\begin{eqnarray*}
\mu(\{x\in\cx:\ |T_{\az,b}g(x)|> t\})&&\ls  t^{-q_1}
\int_\cx|T_{\az,b}g(y)|^{q_1}\,d\mu(y)
\ls t^{-q_1}\|g\|^{q_1}_{\lpo}\\
&&\ls t^{-q_1} t^{q_0(p_1-1)q_1/p_1}\|f\|^{q_1/p_1}_{\lon}\ls t^{-q_0},
\end{eqnarray*}
where $T_{\az,b}:=T_{\az,b_1}$. On the other hand, by \eqref{2.3} with $p=1$ and
$t$ replaced by $t^{q_0}$, and the fact that the sequence
of balls, $\{B_j\}_j$, is pairwise disjoint, we see that $\mu\lf(\cup_j6^2B_j\r)\ls
 t^{-q_0}\int_\cx|f(y)|\,d\mu(y)\ls t^{-q_0}$,
and hence the proof
of Step i) can be reduced to proving
\begin{eqnarray}\label{3.14}
&&\mu\lf(\lf\{x\in\cx\setminus\lf(\bigcup_j6^2B_j\r):\ |T_{\az,b}h(x)|> t\r\}\r)\\
&&\noz\hs\ls\lf[\|\Phi_{1/r}( t^{-1}|f|)\|_{\lon}
+\Phi_{1/r}( t^{-1}\|f\|_{\lon})\r]^{q_0}.
\end{eqnarray}

For each fixed $j$ and all $x\in\cx$, let $b_j(x):=b(x)-m_{\wz B_j}(b)$ and write
$$T_{\az,b}h(x)=\sum_jb_j(x)T_\az h_j(x)-\sum_jT_\az(b_jh_j)(x)
=:\mathrm{I}(x)+\mathrm{II}(x).$$

For the term $\mathrm{II}(x)$, by the boundedness of $T_\az$
from $\lon$ to $L^{q_0,\fz}(\mu)$,
we conclude that
\begin{eqnarray*}
\mu\lf(\{x\in \cx\,:\,|\mathrm{II}(x)|> t\}\r)
&&\ls  t^{-q_0}
\lf[\sum_j\int_\cx|b_j(y)h_j(y)|\,d\mu(y)\r]^{q_0}\\
&&\ls t^{-q_0}\lf[\sum_j\int_\cx|b(y)-m_{\wz B_j}(b)|
|f(y)|\oz_j(y)\,d\mu(y)\r]^{q_0}\\
&&\hs+ t^{-q_0}\lf[\sum_j\|\vz_j\|_{\li}\int_{R_j}|b(y)-m_{\wz B_j}(b)|\,d\mu(y)\r]^{q_0}
=:\rm{U}+\rm{V}.
\end{eqnarray*}

By Lemma \ref{l2.6}(iii), we easily know that $R_j$ is also $(6,\bz_6)$-doubling and $R_j=\wz R_j$. Thus,
from Lemmas \ref{l2.2} and \ref{l2.1}, an argument similar to
that used in the proof of \cite[Theorem 1.2]{hmy1},
\eqref{2.5} and the fact that
$\{6B_j\}_j$ is a sequence of finite overlapping balls, we deduce that
\begin{equation}\label{3.15}
\mathrm{V}\ls t^{-q_0}\lf[\sum_j\|\vz_j\|_{\li}\mu(R_j)\r]^{q_0}
\ls t^{-q_0}\lf[\int_\cx|f(y)|\,d\mu(y)\r]^{q_0}
\end{equation}

On the other hand, by the generalized H\"older inequality
(\cite[Lemma 4.1]{fyy}), Lemma \ref{l2.2} and an argument
similar to that used in the proof of \cite[Theorem 1.2]{hmy1}, we have
\begin{equation}\label{3.16}
\mathrm{U}\ls\lf[\|\Phi_{1/r}( t^{-1}|f|)\|_{\lon}
+\Phi_{1/r}( t^{-1}\|f\|_{\lon})\r]^{q_0}.
\end{equation}

Combining \eqref{3.15} and \eqref{3.16}, we know that
\begin{equation}\label{3.17}
\mu\lf(\{x\in\cx\,:\,|\mathrm{II}(x)|> t\}\r)
\ls\lf[\|\Phi_{1/r}( t^{-1}|f|)\|_{\lon}
+\Phi_{1/r}( t^{-1}\|f\|_{\lon})\r]^{q_0},
\end{equation}
which is desired.

Now we turn our attention to $\mathrm{I}(x)$. Let $x_j$ be the center of $B_j$.
Let $\tz$ be a bounded
function with $\|\tz\|_{L^{q_0'}(\mu)}\le 1$ and the support contained in
$\cx\setminus(\cup_j6^2B_j)$. By the vanishing moment of $h_j$ and \eqref{1.8},
we see that
\begin{eqnarray*}
&&\int_{\cx\setminus(\cup_j6^2B_j)}|\mathrm{I}(x)\tz(x)|\,d\mu(x)\\
&&\hs\ls\sum_j\int_{\cx\setminus 2R_j}|b_j(x)\tz(x)|
\lf|\int_{\cx}h_j(y)[K_\az(x,y)-K_\az(x,x_j)]\,d\mu(y)\r|\,d\mu(x)\\
&&\hs\hs+\sum_j\int_{2R_j\setminus6^2B_j}|b_j(x)\tz(x)||T_\az h_j(x)|\,d\mu(x)\\
&&\hs\ls\sum_jr_{R_j}^\dz\int_\cx|h_j(y)|\,d\mu(y)\int_{\cx\setminus 2R_j}
\frac{|b_j(x)\tz(x)|}{[d(x,x_j)]^\dz[\lz(x_j,d(x,x_j))]^{1-\az}}\,d\mu(x)\\
&&\hs\hs+\sum_j\int_{2R_j\setminus6^2B_j}|b_j(x)\tz(x)||T_\az(\oz_jf)(x)|\,d\mu(x)\\
&&\hs\hs+\sum_j\int_{2R_j}|b_j(x)\tz(x)||T_\az(\vz_j)(x)|\,d\mu(x)
=:\mathrm{G}+\mathrm{H}+\mathrm{J}.
\end{eqnarray*}
From \eqref{1.2}, H\"older's inequality, Corollary \ref{c2.3},
\eqref{2.1}, (i) through (iv) of Lemma \ref{l2.1}, we deduce that
\begin{eqnarray*}
&&\int_{\cx\setminus 2R_j}
\frac{|b_j(x)\tz(x)|}{[d(x,x_j)]^\dz[\lz(x_j,d(x,x_j))]^{1-\az}}\,d\mu(x)\\
&&\hs\ls\sum_{k=1}^\fz\lf(2^kr_{R_j}\r)^{-\dz}
\frac 1{[\lz(x_j,2^kr_{R_j})]^{1-\az}}\int_{2^{k+1}R_j}|b(x)-m_{\wz{2^{k+1}R_j}}(b)|
|\tz(x)|\,d\mu(x)\\
&&\hs\hs+\sum_{k=1}^\fz\lf(2^kr_{R_j}\r)^{-\dz}\frac
1{[\lz(x_j,2^kr_{R_j})]^{1-\az}}
|m_{\wz{B_j}}(b)-m_{\wz{2^{k+1}R_j}}(b)|\int_{2^{k+1}R_j}|\tz(x)|\,d\mu(x)\\
&&\hs\ls\sum_{k=1}^\fz\lf(2^kr_{R_j}\r)^{-\dz}\lf[\frac
{\mu(2^{k+2}R_j)}{\lz(x_j,6^kr_{R_j})}\r]^{1-\az}
+\sum_{k=1}^\fz K_{\wz{B_j},\wz{2^{k+1}R_j}}
\lf(2^kr_{R_j}\r)^{-\dz}
\lf[\frac{\mu(6^{k+1}R_j)}{\lz(x_j,6^kr_{R_j})}\r]^{1-\az}\\
&&\hs\ls r_{R_j}^{-\dz},
\end{eqnarray*}
where we used the fact that
$$K_{\wz{B_j},\wz{2^{k+1}R_j}}\ls K_{\wz{B_j},R_j}+K_{{R_j},2^{k+1}R_j}
+K_{2^{k+1}R_j,\wz{2^{k+1}R_j}}
\ls K_{{R_j},2^{k+1}R_j}\ls k.$$
Since $\|h_j\|_{\lon}\ls \int_\cx|f(y)|\oz_j(y)\,d\mu(y)$, we further see that
${\rm G}\ls\|f\|_{\lon}$.

On the other hand, applying H\"older's inequality, Corollary \ref{c2.3}, \eqref{2.1},
(iv), (i) and (iii) of Lemma \ref{l2.1}, the boundedness
of $T_\az$ from $\lpo$ to $\lqo$ with $p_1\in(p_0,1/\az)$ and $1/q_1=1/p_1-\az$, \eqref{2.7},
and the fact that $\{6Q_j\}_j$ is a sequence of finite
overlapping balls, we obtain
\begin{eqnarray*}
\mathrm{J}&&\le\sum_j\int_{2R_j}\lf[|b(x)-m_{\wz{2R_j}}(b)|
+|m_{\wz{B_j}}(b)-m_{\wz{2R_j}}(b)|\r]|T_\az(\vz_j)(x)\tz(x)|\,d\mu(x)\\
&&\le\|\tz\|_{L^{q'_0}(\mu)}
\sum_j\lf\{\lf[\int_{2R_j}|b(x)-m_{\wz{2R_j}}(b)|^{q_0}|T_\az\vz_j(x)|^{q_0}
\,d\mu(x)\r]^{1/q_0}\r.\\
&&\lf.\hs+\lf[\int_{2R_j}|T_\az\vz_j(x)|^{q_0}\,d\mu(x)\r]^{1/q_0}
\lf|m_{\wz{B_j}}(b)-m_{\wz{2R_j}}(b)\r|\r\}\\
&&\ls\sum_j\lf\{\|T_\az\vz_j\|_{L^{q_1}(\mu)}\lf[\int_{2R_j}
|b(x)-m_{\wz{2R_j}}(b)|^{q_0(q_1/q_0)'}\,d\mu(x)\r]^{1/q_0-1/q_1}
\r.\\
&&\hs+[\mu(4R_j)]^{1/q_0-1/q_1}|m_{\wz{B_j}}(b)-m_{\wz{2R_j}}(b)|\Bigg\}
\ls\sum_j[\mu(4R_j)]^{1/q_0-1/q_1}\|\vz_j\|_{L^{p_1}(\mu)}\\
&&\ls\sum_j[\mu(4R_j)]^{1/q_0-1/q_1}\|\vz_j\|_{\li}[\mu(R_j)]^{1/p_1}
\ls\int_\cx|f(x)|\,d\mu(x),
\end{eqnarray*}
where we used the fact that
$$|m_{\wz{B_j}}(b)-m_{\wz{2R_j}}(b)|\le|m_{\wz{B_j}}(b)-m_{R_j}(b)|
+|m_{R_j}(b)-m_{\wz{2R_j}}(b)|\ls 1.$$

To estimate \rm{H}, by \eqref{1.7}, \eqref{1.2} and \eqref{1.3}, we see that, for
all $x\in 2R_j\setminus6^2B_j$,
$$|T_\az(\oz_jf)(x)|\ls \frac 1{[\lz(x_j,d(x,x_j))]^{1-\az}}
\int_{6B_j}|f(y)|\oz_j(y)\,d\mu(y),$$
which further implies that
\begin{eqnarray*}
\mathrm{H}&&\ls\sum_j\lf\{\int_{2R_j\setminus R_j}
\frac{|b_j(x)\tz(x)|}{[\lz(x_j,d(x,x_j))]^{1-\az}}\,d\mu(x)
+\int_{R_j\setminus 6^2B_j}\cdots\r\}\int_\cx|f(y)|\oz_j(y)\,d\mu(y)\\
&&\ls\sum_j\lf\{\frac1{[\lz(x_j,r_{R_j})]^{1-\az}}\lf[\int_\cx|b_j(x)|^{q_0}
\,d\mu(x)\r]^{1/q_0}+\sum_{k=0}^{N-1}
\lf[\frac {\mu((3\times6^2)^{k+2}B_j)}{\lz(x_j,(3\times6^2)^kr_{B_j})}\r]^{1-\az}\r.\\
&&\hs\lf.+\sum_{k=0}^{N-1}
\lf[\frac{\mu((3\times6^2)^{k+1}B_j)}{\lz(x_j,(3\times6^2)^kr_{B_j})}\r]^{1-\az}
|m_{\wz{B_j}}(b)-m_{\wz{(3\times6^2)^{k+1}B_j}}(b)|\r\}\int_\cx|f(y)|\oz_j(y)\,d\mu(y),
\end{eqnarray*}
where $N\in\nn$ satisfies that $R_j=(3\times6^2)^N B_j$. Obviously, for each $k\in\{0,\ldots,N-1\}$,
$(3\times6^2)^kB_j\st R_j$ and hence
$$|m_{\wz{B_j}}(b)-m_{\wz{(3\times6^2)^{k+1}B_j}}(b)|\ls K_{B_j,(3\times6^2)^{k+1}B_j}
\ls K_{B_j,R_j}\ls 1.$$
Consequently, by the fact that $R_j$ is the smallest
$(3\times6^2,C_\lz^{(3\times 6^2)+1})$-doubling ball of the family
$\{(3\times6^2)^kB_j\}_{k\in\nn}$
and an argument similar to that used in
the proof of Lemma \ref{l3.4}(iii), we see that
\begin{eqnarray*}
\mathrm{H}&&\ls\sum_j\lf(1+\sum_{k=0}^{N-1}\lf[\frac{\mu((3\times6^2)^kB_j)}
{\lz(x_j,(3\times6^2)^kr_{B_j})}\r]^{1-\az}\r)
\int_\cx|f(y)|\oz_j(y)\,d\mu(y)\ls\int_\cx|f(y)|\,d\mu(y).
\end{eqnarray*}

Combining the estimates for \rm{G}, \rm{H} and \rm{J}, we then conclude that
$$\int_{\cx\setminus(\cup_j(3\times6^2)^2B_j)}|\mathrm{I}(x)\tz(x)|
\,d\mu(x)\ls\|f\|_{\lon}.$$
Thus, we have
\begin{eqnarray*}
&&\mu\lf(\lf\{x\in\cx\setminus\lf(\bigcup_j6^2B_j\r):\ |\mathrm{I}(x)|>t\r\}\r)\\
&&\noz\hs\ls t^{-q_0}\int_{\cx\setminus(\cup_j6^2B_j)}|\mathrm{I}(x)|^{q_0}\,d\mu(x)
\ls\lf[ t^{-1}\int_{\cx\setminus(\cup_j6
^2B_j)}|f(x)|\,d\mu(x)\r]^{q_0},
\end{eqnarray*}
which, together with \eqref{3.17}, implies \eqref{3.14} and hence completes the
proof of Theorem \ref{t1.19} in the case that $k=1$.

Step ii) $k>1$. The proof of this case is completely analogous to that of
\cite[Theorem 1.2]{hmy1},
the details being omitted, which completes the proof of Theorem \ref{t1.19}.
\end{proof}

\section{Some applications\label{s4}}

\hskip\parindent In this section, we apply all the results of Theorems \ref{t1.13},
\ref{t1.15} and \ref{t1.19} to a specific example of fractional
integrals to obtain some interesting conclusions.

We first need the following notion.

\begin{definition}
Let $\ez\in(0,\fz)$. A dominating function $\lz$ is said to satisfy the
\emph{$\ez$-weak reverse doubling condition} if, for all $r\in(0, 2\,{\mathop\mathrm{diam}}(\cx))$
and $a\in(1,2\,{\mathop\mathrm{diam}}(\cx)/r)$, there exists a number $C(a)\in[1,\fz)$,
depending only on $a$ and $\cx$, such that, for all
$x\in\cx$,
\begin{equation}\label{4.1}
\lz(x,ar)\ge C(a)\lz(x,r)
\end{equation}
and, moreover,
\begin{equation}\label{4.2}
\sum_{k=1}^{\fz}\frac1{[C(a^k)]^{\ez}}<\fz.
\end{equation}
\end{definition}

\begin{remark}\label{r4.1}
(i) We remark that the  {$1$-weak reverse doubling condition} is just the
weak reverse doubling condition introduced in \cite[Definition 3.1]{fyy2}. Moreover,
it is easy to see that, if $\ez_1<\ez_2$ and $\lz$ satisfies
the  {$\ez_1$-weak reverse doubling condition}, then $\lz$ also satisfies
the  {$\ez_2$-weak reverse doubling condition}.

(ii) Assume that $\diam(\cx)=\fz$.
Let $a=2^k$ and $r=2^{-k}$ in \eqref{4.1}. Then, by \eqref{4.2}, we see that,
for any fixed $x\in\cx$,
$$\lim_{k\to\fz}\lz(x,2^{-k})\le\lim_{k\to\fz}\frac1{C(2^k)}\lz(x,1)=0.
$$
Thus, by the fact that $r\to\lz(x,r)$
is non-decreasing  for any fixed $x\in\cx$,
we further know that $\lim_{r\to 0}\lz(x,r)=0$.

On the other hand, by \eqref{4.2}, we see that $\lim_{k\to\fz}C(2^k)=\fz$.
Letting $a=2^k$
and $r=1$ in \eqref{4.1}, by an argument similar to the case
$r\to 0$, we  know that, for any fixed $x\in\cx$, $\lim_{r\to\fz}\lz(x,r)=\fz$.

(iii) By Remark \ref{r1.4}(i), the
dominating function in the Euclidean space $\rd$ with
a Radon measure $\mu$ as in \eqref{1.1} is $\lz(x,r):=C_0r^{\kz}$, which
satisfies the {$\ez$-weak reverse doubling condition} for any
$\ez\in(0,\fz)$.

(iv) If $(\cx, d, \mu)$ is an RD-\emph{space}, namely, a
space of homogeneous type in the
sense of Coifman and Weiss with a measure
$\mu$ satisfying both the doubling and the reverse doubling conditions,
then  $\lz(x,r):=\mu(B(x,r))$
is the dominating function satisfying
the  {$\ez$-weak reverse doubling condition}
for any $\ez\in(0,\fz)$.
It is known that a  connected space of homogeneous type in the
sense of Coifman and Weiss is always an RD-space
(see \cite[p.\,65]{yz} and
\cite[Remark 3.4(ii)]{fyy2}).

(v) We remark that the {$\ez$-weak reverse doubling condition} is much weaker than
the assumption introduced by Bui and Duong in \cite[Subsection 7.3]{bd}:
there exists $m\in(0,\fz)$ such that, for all
$x\in\cx$ and $a,\,r\in (0,\fz)$,
$\lz(x,ar)=a^m\lz(x,r).$
\end{remark}

Before we give an example, we first establish a technical
lemma adapted from \cite[Lemma 2.1]{gg}.  It turns out that the integral
kernel $1/{[\lambda(y,d(x,y))]^{1-\az}}$ for $\az\in(0,1)$ is locally integrable.

\begin{lemma}\label{l4.2}
Let $\alpha\in (0,1)$ and $\lz$ satisfy the {$\az$-weak reverse doubling condition}.
Then there exists a positive constant $C$, depending
on $\az$ and $m$, such that, for all $x\in\cx$ and $r\in(0,2\,{\mathop\mathrm{diam}}(\cx))$,
$$\int_{B(x,r)}\frac 1{[\lz(y,d(x,y))]^{1-\az}}\,d\mu(y)\le C[\lz(x,r)]^\az.$$
\end{lemma}

\begin{proof}
From \eqref{1.3}, \eqref{1.2}, \eqref{4.1} and \eqref{4.2}, we deduce that
\begin{eqnarray*}
&&\int_{B(x,r)}\frac 1{[\lz(y,d(x,y))]^{1-\az}}\,d\mu(y)\\
&&\hs\ls \int_{B(x,r)}\frac 1{[\lz(x,d(x,y))]^{1-\az}}\,d\mu(y)\ls\sum_{j=0}^\fz
\frac{\mu(B(x,2^{-j}r))}{[\lz(x,2^{-j-1}r)]^{1-\az}}
\ls\sum_{j=0}^{\fz}\frac{\lz(x,2^{-j}r)}{[\lz(x,2^{-j-1}r)]^{1-\az}}\\
&&\hs\ls\sum_{j=0}^{\fz}[\lz(x,2^{-j-1}r)]^\az
\ls\sum_{j=1}^{\fz}\frac1{[C(2^j)]^\az}[\lz(x,r)]^\az
\ls[\lz(x,r)]^\az,
\end{eqnarray*}
which completes the proof of Lemma \ref{l4.2}.
\end{proof}

For all $\az\in(0,1)$, $f\in L^{\fz}_b(\mu)$ and $x\in\cx$, the
\emph{fractional integral} $I_\az f(x)$ is defined by
\begin{eqnarray}\label{4.3}
I_\az f(x):=\int_\cx \frac{f(y)}{[\lambda(y,d(x,y))]^{1-\az}}\,d\mu(y).
\end{eqnarray}
Notice that, if $(\cx,d,\mu)=(\rr^d,|\cdot|,\mu)$,
$\lz(x,r)=r^\kz$ with $\kz\in(0,d]$ and the measure $\mu$ is as in \eqref{1.1},
then $I_\az$ is just the classical fractional integral in the non-doubling space
$(\rr^d,|\cdot|,\mu)$.

We now show that the kernel of $I_\az$ satisfies all the
assumptions of this article. By \eqref{1.3}, we know that the
integral kernel $K_\az(x,y):=\frac 1{[\lambda(y,d(x,y))]^{1-\az}}$ satisfies
\eqref{1.7}.
By Remark \ref{1.4}(iii), without loss of generality,
we may assume that $\lz$ satisfies that there exist
$\ez,\,\wz{C}\in(0,\fz)$ such that,
for all $x\in\cx$, $r\in(0,\fz)$ and $t\in[0,r]$,
\begin{equation}\label{4.4}
|\lz(x,r+t)-\lz(x,r)|
\le \wz{C}\frac{t^\ez}{r^\ez}\lz(x,r).
\end{equation}

\begin{remark}\label{r4.3}
By \eqref{4.4}, we see that, for a fixed $x\in\cx$, $r\to\lz(x,r)$ is
continuous on $(0,\fz)$.
\end{remark}
Now we
show that the integral kernel $K_\az$ of $I_\az$ also satisfies
\eqref{1.8}.

\begin{proposition}\label{p4.4}
Assume that $\lz$ satisfies \eqref{4.4}. Then
the integral kernel $K_\az$ of $I_\az$ in \eqref{4.3} satisfies
\eqref{1.8}.
\end{proposition}

\begin{proof}
For all $x,\,\wz{x},\,y\in\cx$ with $d(x,y)\ge 2d(x,\wz{x})$,
we consider the following two cases.

Case i) $d(x,y)\le d(\wz{x},y)$. Let $t=d(\wz{x},y)-d(x,y)$ and $r=d(x,y)$.
Then, by $0\le t\le d(x,\wz{x})\le\frac12 d(x,y)\le d(x,y)=r$ and \eqref{4.4},
we see that
\begin{equation*}
|\lz(y,d(\wz{x},y))-\lz(y,d(x,y))|\ls\frac{[d(\wz{x},y)-d(x,y)]^\ez}{[d(x,y)]^\ez}
\lz(y,d(x,y))\ls\lf[\frac{d(x,\wz{x})}{d(x,y)}\r]^\ez\lz(y,d(x,y)).
\end{equation*}
From this, $d(x,y)\le d(\wz{x},y)$, Definition \ref{d1.3} and
\eqref{1.3}, we further deduce that
\begin{eqnarray*}
&&|K_\az(x,y)-K_\az(\wz{x},y)|\\
&&\hs\le\lf|\frac1{\lz(y,d(x,y))}-\frac1{\lz(y,d(\wz{x},y))}\r|^{1-\az}
=\frac{|\lz(y,d(\wz{x},y))-\lz(y,d(x,y))|^{1-\az}}
{[\lz(y,d(\wz{x},y))\lz(y,d(x,y))]^{1-\az}}\\
&&\hs\ls\frac{[d(x,\wz{x})]^{\ez(1-\az)}}
{[d(x,y)]^{\ez(1-\az)}[\lz(y,d(\wz{x},y))]^{1-\az}}
\ls\frac{[d(x,\wz{x})]^{\ez(1-\az)}}
{[d(x,y)]^{\ez(1-\az)}[\lz(x,d(x,y))]^{1-\az}}.
\end{eqnarray*}
This finishes the proof of \eqref{1.8} in this case.

Case ii) $d(\wz{x},y)\le d(x,y)$. In this case, since $d(x,y)\ge 2d(x,\wz{x})$,
it follows that $$d(x,\wz{x})\le \frac12d(x,y)
\le \frac12[d(x,\wz{x})+d(\wz{x},y)],$$
and hence $d(x,\wz{x})\le d(\wz{x},y)$.
Then, by an argument similar to that used in the proof of Case i),
we see that
\begin{eqnarray*}
&&|K_\az(x,y)-K_\az(\wz{x},y)|\ls\frac{[d(x,\wz{x})]^{\ez(1-\az)}}
{[d(\wz{x},y)]^{\ez(1-\az)}[\lz(x,d(x,y))]^{1-\az}},
\end{eqnarray*}
which, together with $d(x,y)\le d(x,\wz{x})+d(\wz{x},y)\le 2d(\wz{x},y)$,
further implies that  \eqref{1.8} holds true in this case.
This finishes the proof of Proposition \ref{p4.4}.
\end{proof}

To consider the boundedness of $I_\az$ on Lebesgue spaces, we need
the following  Welland inequality in the present setting,
which is a variant of
\cite[Theorem 6.4]{gm}.

\begin{lemma}\label{l4.5}
Assume that $\diam (\cx)=\fz$.
Let $\az\in(0,1)$, $\ez\in(0,\min\{\az,1-\az\})$ and
$\lz$ satisfy the {$\ez$-weak reverse doubling condition}.
Then there exists a positive constant $C$, independent of $f$ and $x$, such that,
for all $x\in\cx$ and $f\in L^{\fz}_b(\mu)$,
$$|I_\az f(x)|\le C\lf[M_{1,6}^{(\az+\ez)}f(x)M_{1,6}^{(\az-\ez)}f(x)\r]^{1/2}, $$
where $M_{1,6}^{(\az)}$ for $\az\in(0,1)$ is defined as in Lemma \ref{l3.1}.
\end{lemma}

\begin{proof}
Without loss of generality,
we may assume that the right-hand side of the desired
inequality is finite.
Let $s\in (0,\fz)$. We write
$$|I_\az f(x)|\le\int_{B(x,s)}\frac{|f(y)|}{[\lambda(y,d(x,y))]^{1-\az}}\,d\mu(y)
+\int_{\cx\setminus B(x,s)}\cdots=:\rm I+\rm II.$$
By \eqref{1.3}, \eqref{1.2}, \eqref{4.1} and \eqref{4.2}, we see that,
\begin{eqnarray*}
\rm I&&\ls\int_{B(x,s)}\frac {|f(y)|}{[\lz(x,d(x,y))]^{1-\az}}\,d\mu(y)
\ls\sum_{j=0}^\fz\frac 1{[\lz(x,2^{-j-1}s)]^{1-\az}}\int_{B(x,2^{-j}s)}|f(y)|
\,d\mu(y)\\
&&\sim\sum_{j=0}^\fz
\frac{[\lz(x,2^{-j-1}s)]^\ez}{[\lz(x,2^{-j-1}s)]^{1-\az+\ez}}
\int_{B(x,2^{-j}s)}|f(y)|\,d\mu(y)\\
&&\ls[\lz(x,s)]^\ez \sum_{j=1}^{\fz}\frac1{[C(2^{j})]^\ez}M_{1,6}^{(\az-\ez)}f(x)
\ls[\lz(x,s)]^\ez M_{1,6}^{(\az-\ez)}f(x).
\end{eqnarray*}

Similarly, we also see that ${\rm II}
\ls[\lz(x,s)]^{-\ez}M_{1,6}^{(\az+\ez)}f(x)$. Thus,
$$|{\rm I}_\az f(x)|\ls[\lz(x,s)]^\ez M_{1,6}^{(\az-\ez)}f(x)
+[\lz(x,s)]^{-\ez}M_{1,6}^{(\az+\ez)}f(x).$$
By Remark \ref{r4.1}(ii) and Remark \ref{r4.3}, we can choose $s\in (0,\fz)$ such that $$[\lz(x,s)]^{\ez}:=\lf[\frac{M_{1,6}^{(\az+\ez)}f(x)}
{M_{1,6}^{(\az-\ez)}f(x)}\r]^{1/2}.$$
Then we obtain the desired conclusion and hence complete the proof of Lemma \ref{l4.5}.
\end{proof}

Now we are ready to state the main theorem of this section.

\begin{theorem}\label{t4.6}
 Assume that   $\diam(\cx)=\fz$.
 Let $\az\in(0,1)$, $p\in(1,1/\az)$ and $1/q=1/p-\az$.
If $\lz$ satisfies the
{$\ez$-weak reverse doubling condition}
for some $\ez\in(0,\min\{\az,1-\az,1/q\})$, then $I_\az$ is bounded from $\lp$
into $\lq$.
\end{theorem}

\begin{proof}
Let  $\frac1{q_\ez^+}:=\frac1q-\ez$,
$\frac1{q_\ez^-}:=\frac1q+\ez$, $q^+:=2\frac{q_\ez^+}q$ and $q^-:=2\frac{q_\ez^-}q$.
Then we have
$1<p<q_\ez^-<q<q_\ez^+<\fz$,  $1<q^-<q^+<\fz$ and $1/q^++1/q^-=1$.
From Lemma \ref{l4.5}, H\"older's inequality and Lemma \ref{l3.1}, it follows that
\begin{eqnarray*}
\|I_\az f\|_\lq&&\ls\lf\|\lf[M_{1,6}^{(\az+\ez)}f\r]^{q/2}\r\|^{1/q}_{L^{q^+}(\mu)}
\lf\|\lf[M_{1,6}^{(\az-\ez)}f\r]^{q/2}\r\|^{1/q}_{L^{q^-}(\mu)}\\
&&\sim\|M_{1,6}^{(\az+\ez)}f\|^{1/2}_{L^{q_\ez^+}(\mu)}
\|M_{1,6}^{(\az-\ez)}f\|^{1/2}_{L^{q_\ez^-}(\mu)}\ls\|f\|_\lp^{1/2}\|f\|_\lp^{1/2}
\sim\|f\|_\lp,
\end{eqnarray*}
which completes the proof of Theorem \ref{t4.6}.
\end{proof}

From Theorems \ref{t4.6}, \ref{t1.13}, \ref{t1.15} and
\ref{t1.19}, we immediately deduce the following interesting conclusions,
the details being omitted.

\begin{corollary}\label{c4.7}
Under the same assumption as that of Theorem \ref{t4.6},
all the conclusions of
Theorems \ref{t1.13},
\ref{t1.15} and \ref{t1.19} hold true, if $T_\az$ therein is replaced by $I_\az$
as in \eqref{4.3}.
\end{corollary}

\medskip

\noindent{\bf Acknowledgements.} The first author would like to
express his deep thanks to Professors Yan Meng, Haibo Lin and Dongyong Yang
for some helpful discussions on the subject of this paper.

\bigskip

Xing Fu, Dachun Yang (Corresponding author) and Wen Yuan

\medskip

School of Mathematical Sciences, Beijing Normal University,
Laboratory of Mathematics and Complex Systems, Ministry of
Education, Beijing 100875, People's Republic of China

\smallskip

{\it E-mails}: \texttt{xingfu@mail.bnu.edu.cn} (X. Fu)

\hspace{1.55cm}\texttt{dcyang@bnu.edu.cn} (D. Yang)

\hspace{1.55cm}\texttt{wenyuan@bnu.edu.cn} (W. Yuan)
\end{document}